\documentclass[10pt,journal]{IEEEtran}

\usepackage{amsmath,amsthm,amssymb,float,graphicx,geometry}
\usepackage{bbm,bbding,amssymb,pifont,mathrsfs,amsfonts,graphicx,subfigure,mathtools,color}%
\usepackage{amscd,array,enumerate,dsfont,texdraw,tikz,multicol,bbm}
\usepackage{algorithmic}
\newtheorem{algorithm}{Algorithm}
\usepackage{algorithm}
\usepackage{footnote}
\usepackage{lipsum}

\allowdisplaybreaks[4]
\usetikzlibrary{shapes.geometric, arrows}

\newtheorem{assumption}{Assumption}
\newtheorem{lem}{Lemma}
\newtheorem{thm}{Theorem}
\newtheorem{Def}{Definition}
\newtheorem{Col}{Corollary}
\newtheorem{remark}{Remark}

\ifCLASSINFOpdf
\else
\fi

\begin{document}
\title{No-regret distributed learning in subnetwork zero-sum games}
\author{Shijie~Huang\textsuperscript{a,b},
	Jinlong~Lei\textsuperscript{c},
	Yiguang~Hong\textsuperscript{a,c},
	Uday~V.~Shanbhag\textsuperscript{d},
	and Jie~Chen\textsuperscript{c}
\thanks{
This work was supported by Shanghai Municipal Science and Technology Major Project (2021SHZDZX0100) and the National Natural Science Foundation of China (61733018).

\noindent\textsuperscript{a} Key Laboratory of Systems and Control, Academy of Mathematics and Systems Science, Chinese Academy of Sciences, Beijing 100190, China;

\noindent\textsuperscript{b} School of Mathematical Sciences, University of Chinese Academy of Sciences, Beijing 100049, China; 

\noindent\textsuperscript{c} Department of Control Science and Engineering \& Shanghai Research Institute for Intelligent Autonomous Systems, Tongji University, Shanghai 201804, China; 

\noindent\textsuperscript{d} Department of Industrial and Manufacturing Engineering, Pennsylvania State University, University Park, PA 16802, USA.

Email address: \texttt{sjhuang@amss.ac.cn} (S. Huang),
 \texttt{leijinlong@tongji.edu.cn} (J.~Lei), \texttt{yghong@iss.ac.cn} (Y.~Hong), \texttt{udaybag@psu.edu} (U.~V.~Shanbhag), \texttt{chenjie@bit.edu.cn} (J.~Chen).

}}
%
\maketitle

\begin{abstract}
In this paper, we consider a distributed learning problem in a subnetwork zero-sum game, where agents are competing in different subnetworks. These agents are connected through time-varying graphs where each agent has its own cost function and can receive information from its neighbors. We propose a distributed mirror descent algorithm for computing a Nash equilibrium and establish a sublinear regret bound on the sequence of iterates when the graphs are uniformly strongly connected and the cost functions are convex-concave. Moreover, we prove its convergence with suitably selected diminishing step-sizes for a strictly convex-concave cost function. We also consider a constant step-size variant of the algorithm and establish an asymptotic error bound between the cost function values of running average actions and a Nash equilibrium. In addition, we apply the algorithm to compute a mixed-strategy Nash equilibrium in subnetwork zero-sum finite-strategy games, which have merely convex-concave (to be specific, multilinear) cost functions, and obtain a final-iteration convergence result and an ergodic convergence result, respectively, under different assumptions.
\end{abstract}


\IEEEpeerreviewmaketitle

\section{INTRODUCTION}
A non-cooperative game is a framework for studying the interaction of agents whose decisions are affected by the actions of others. Zero-sum games represent an important class of non-cooperative games, and some complex decision-making problems, such as power allocation in wireless communication \cite{boyd2004convex}, robust portfolio selection in finance \cite{goldfarb2003robust},  and robust matched filtering in signal processing \cite{kassam1985robust} may be modeled by generalized zero-sum games, named as ''subnetwork zero-sum'' \cite{lou2015nash}, where two subnetworks of agents are engaged in a zero-sum game. A core equilibrium concept of non-cooperative games is the Nash equilibrium (NE) and the NE seeking problems have been extensively studied.

Conventional NE seeking algorithms mainly deal with full-decision information scenario, where each agent may observe the actions of all of its rivals. However, in practice, agents may have to make decisions based only on limited information, suggesting the design of distributed NE seeking algorithms. For example, in \cite{Koshal}, distributed synchronous and asynchronous algorithms with diminishing step-sizes have been provided with an error bound for a constant step-size for strictly aggregative games, while in \cite{gadjov2020single}, a fully-distributed generalized NE seeking algorithm for aggregative games has been further designed based on an operator splitting scheme. Moreover, in \cite{lei2020distributed} and \cite{pang2020distributed}, a regularized distributed algorithm for merely monotone aggregative games and a gradient-free distributed algorithm for convex games with limited cost function knowledge have been proposed, respectively. Distinct from these directions, a linearly convergent distributed gradient-response scheme for stochastic aggregative games has been introduced in \cite{Lei_CDC}. In addition, continuous-time distributed algorithms via consensus-based approaches have also been analyzed in \cite{Liang} and \cite{Ye2017}.

Much of the aforementioned work focuses on whether those proposed algorithms guarantee the convergence to a NE. In game-theoretic learning algorithms, there exists a class of dynamics that focuses on the learning process, called no-regret dynamics \cite{Roughgarden}. A no-regret learning process is a natural choice for the agents since no player wants to realize that the action sequence he/she employed is strictly inferior to taking some fixed action in hindsight. In \cite{daskalakis2011near}, the authors used Nesterov's excessive gap technique to propose a near-optimal no-regret algorithm which can ensure that the cost value converges to the cost value of a NE with rate $O(\frac{(\ln T)^{\frac{3}{2}}}{T})$ for two-player zero-sum games with finite action sets. In \cite{rakhlin2013optimization}, this rate was improved to $O(\frac{\ln T}{T})$ by introducing a modified optimistic mirror descent algorithm. Similarly, \cite{kangarshahi2018let} further proposed an optimal no-regret algorithm with a rate $O(\frac{1}{T})$. Moreover, no-regret learning algorithms can only converge to a set of coarse correlated equilibria in a general game, while \cite{mertikopoulos2019learning} provided sufficient conditions under which no-regret learning converges to a NE of the underlying game. Additionally \cite{zhou2018learning} developed an asynchronous no-regret learning algorithm for games with lossy feedback.

However, much of prior research only considers centralized no-regret learning algorithms. Although there are many existing distributed NE seeking algorithms, it is unclear if these algorithms are no-regret schemes \cite{zhou2018learning}. In fact, the only result on distributed online game that we are aware of was provided in \cite{lu2020online}. Inspired by the distributed online optimization problem (\cite{shahrampour2017distributed,yuan2017adaptive}), the authors in \cite{lu2020online} designed an online distributed algorithm to track the generalized NE in dynamic environments and established a sublinear regret bound. Different from the no-regret learning in games that we are considering, \cite{lu2020online} assumed that the cost function is time-varying and the offline benchmark of the regret is the cost value of a NE.

The motivation of this paper is to design effective no-regret distributed algorithms for subnetwork zero-sum games considered in \cite{lou2015nash} and \cite{gharesifard2013distributed}. In the multi-agent network with two subnetworks, agents in the same subnetwork collaborate for consensus, while playing antagonistic roles with the agents in the other subnetwork. The contributions of this paper are summarized as follows:
\begin{itemize}
  \item We propose a distributed learning algorithm based on a mirror descent scheme and derive a regret bound of the algorithm, which shows that the algorithm is no-regret under suitable diminishing and constant step-sizes. Moreover, we prove the convergence to the unique NE when the cost function is strictly convex-concave. To the best of our knowledge, there has been no theoretical result on no-regret distributed algorithms for computing a NE of subnetwork zero-sum games yet.
  \item Although the proposed algorithm cannot guarantee convergence to a NE under constant step-sizes, we obtain an error bound similar to that in \cite{nedic2009subgradient}. This error bound shows that the running average actions generated by the algorithm provide approximate solutions to the NE seeking problem.
  \item We further apply the algorithm to compute a mixed-strategy NE of subnetwork zero-sum finite-strategy games and prove its final-iteration convergence under a restrictive assumption. In addition, we also prove the ergodic convergence of the algorithm different from \cite{mertikopoulos2019learning}, which provides an ergodic convergence result for finite two-person zero-sum games.
\end{itemize}

The remainder of this paper is organized as follows. In Section \uppercase\expandafter{\romannumeral 2}, we formulate the no-regret distributed learning problem in subnetwork zero-sum games and propose a distributed mirror descent algorithm, while in Section \uppercase\expandafter{\romannumeral 3}, we establish several useful lemmas and further give a regret bound analysis of the proposed algorithm. In Section \uppercase\expandafter{\romannumeral 4}, we prove that the algorithm converges to the NE under diminishing step-sizes. We also consider the case of constant step-sizes and establish an asymptotic error bound for the cost value of the averaged iterates. Then, in Section \uppercase\expandafter{\romannumeral 5}, we apply the proposed algorithm to subnetwork zero-sum finite-strategy games and prove two convergence results. In Section \uppercase\expandafter{\romannumeral 6}, we provide simulations to verify our theoretical analysis. Finally, in Section \uppercase\expandafter{\romannumeral 7}, we conclude the paper.

\noindent{\bf Notations}. Denote by $\mathbb{R}^n$ the $n$-dimensional real Euclidean space. For column vectors $x,y\in\mathbb{R}^n$, $\langle x,y\rangle$ denotes the inner product and $\parallel\cdot\parallel_p$ ($p\ge 1$) denotes the $l_p$ norm. For $x\in\mathbb{R}^n$, $\mathbb{B}(x,\epsilon)$ denotes a ball with $x$ the center and $\epsilon > 0$ the radius. For a norm $\|\cdot\|$\footnote{Unless otherwise specified, we use $\|\cdot\|$ to represent any possible $l_p$ ($p\ge 1$) norm in this paper.} on $\mathbb{R}^n$, $\|y\|_{\ast}\triangleq\sup\{\langle y,x\rangle: \|x\|\le 1\}$ denotes the dual norm. $A_{ij}$ denotes the element in the $i$th row and $j$th column of matrix $A$. For a function $f(x_1,\dots,x_N)$, denote $\nabla f$ as the gradient of $f$ and $\partial_i f$ as the subdifferential of $f$ with respect to $x_i$. A function $f(x_1,x_2)$ is said to be (strictly) convex-concave (concave-convex) if $f(x_1,x_2)$ is (strictly) convex (concave) in $x_1$ for any $x_2$ and (strictly) concave (convex) in $x_2$ for any $x_1$. A digraph (directed graph) is characterized by $\mathcal{G} = (\mathcal{V},\mathcal{E})$, where $\mathcal{V} = \{1,\dots,n\}$ is the set of nodes and $\mathcal{E}\subset\mathcal{V}\times\mathcal{V}$ is the set of edges, where $(j,i)\in\mathcal{E}$ if agent $i$ can obtain information from agent $j$. Associated with graph $\mathcal{G}$, there is an adjacency matrix $W = [w_{ij}]\in\mathbb{R}^{n\times n}$ with nonnegative elements, which satisfy that $w_{ij} > 0$ if and only if $(j,i)\in\mathcal{E}$. A path from $i_1$ to $i_p$ is an alternating sequence $i_1e_1\cdots i_{p-1}e_{p-1}i_p$ such that $e_r = (i_r,i_{r+1})\in\mathcal{E}$ $(r = 1,\dots,p-1)$. A digraph is strongly connected if there is a path between any pair of distinct nodes.

\section{PROBLEM FORMULATION AND ALGORITHM}
In this section, we formulate the no-regret distributed learning problem in a subnetwork zero-sum game and propose a distributed mirror descent algorithm.
\subsection{Problem Formulation}
Consider a zero-sum game between two subnetworks $\Sigma_1$ and $\Sigma_2$, composed of agents $\mathcal{V}_1\triangleq\{v_1^1,\dots, v_1^{n_1}\}$ and $\mathcal{V}_2\triangleq\{v_2^1,\dots,v_2^{n_2}\}$, respectively. Assume that the subnetworks $\Sigma_1$ and $\Sigma_2$ are time-varying and described by directed graph sequences $\mathcal{G}_1(k) = (\mathcal{V}_1,\mathcal{E}_1(k))$ and $\mathcal{G}_2(k) = (\mathcal{V}_2,\mathcal{E}_2(k))$.The interaction between $\Sigma_1$ and $\Sigma_2$ is modeled by a bipartite network $\Sigma_{12}$ (Fig. 1). Here $\Sigma_{12}$ is described by a time-varying bipartite graph sequence $\mathcal{G}_{12}(k) = (\mathcal{V}_1\cup\mathcal{V}_2,\mathcal{E}_{12}(k))$, which means that $\mathcal{E}_{12}(k)\subset\{(v_l^i,v_{3-l}^j)\mid v_l^i\in\mathcal{V}_l,v_{3-l}^j\in\mathcal{V}_{3-l},l = 1,2\}$. For $l = 1,2$, and each node $v_l^i\in\mathcal{V}_l$, denote by $\mathcal{N}_l^i(k) \triangleq \{v_l^j\mid(v_l^j,v_l^i)\in\mathcal{E}_l(k)\}$ and $\mathcal{N}_{12,l}^i(k) \triangleq \{v_{3-l}^j\mid (v_{3-l}^j,v_l^i)\in\mathcal{E}_{12}(k)\}$ the set of its neighbors in $\mathcal{V}_l$ and $\mathcal{V}_{3-l}$ at time $k$, respectively.
\begin{figure}[htbp]
\centering
\includegraphics[width = .4\textwidth]{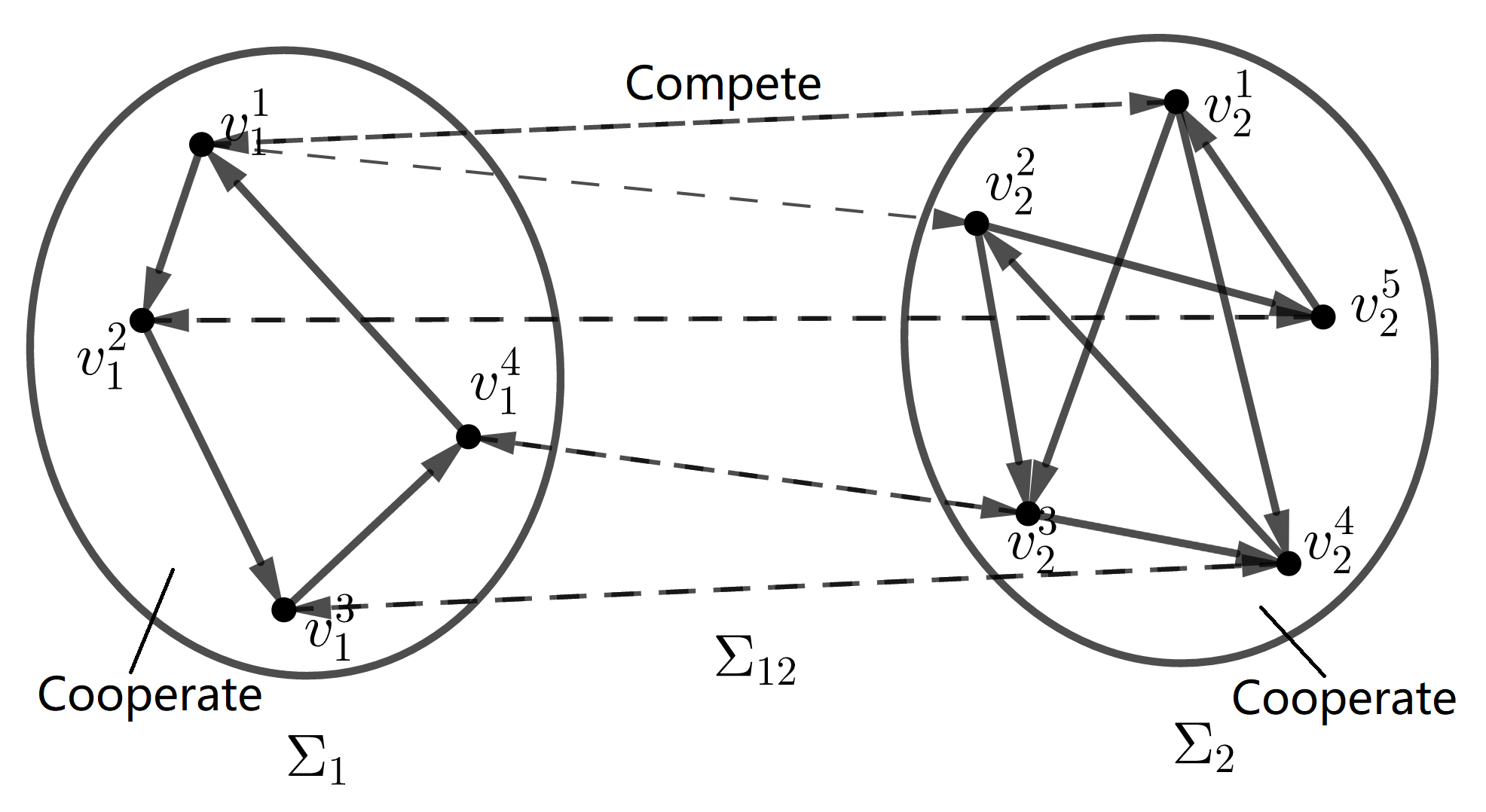}
\caption{A subnetwork zero-sum game}
\end{figure}

For each $l\in\{1,2\}$, the action set of $\Sigma_l$ is denoted by $\mathcal{X}_l\subset\mathbb{R}^{m_l}$. Each subnetwork $\Sigma_l$, $l\in\{1,2\}$ aims to choose an action $x_l\in\mathcal{X}_l$ to minimize the following global cost
\begin{equation}\label{cost_def}
  f_l(x_1,x_2) = \frac{1}{n_l}\sum_{i=1}^{n_l}f_{l,i}(x_1,x_2),\quad l\in\{1,2\}
\end{equation}
which is the average of its agents' costs at each node $v_l^i$, denoted by $f_{l,i}$. The subnetworks are engaged in a zero-sum game, namely, for any $x_l\in\mathcal{X}_l$, $l\in\{1,2\}$,
\[f_1(x_1,x_2) + f_2(x_1,x_2) = 0.\]
The following concept is well known \cite{gharesifard2013distributed}.
\begin{Def}
  An action profile $x^{\ast} = (x_1^{\ast},x_2^{\ast})$ is a Nash equilibrium (NE) of the subnetwork zero-sum game if
  \begin{equation}\label{NE_def}
    x_1^{\ast}\in\arg\min_{x_1\in\mathcal{X}_l}U(x_1,x_2^{\ast}), \text{and}\ x_2^{\ast}\in\arg\max_{x_2\in\mathcal{X}_2}U(x_1^{\ast},x_2),
  \end{equation}
  where $U(x_1,x_2) \triangleq f_1(x_1,x_2)$.
\end{Def}
We make the following assumptions on the action sets and the cost functions $f_{l,i}$.
\begin{assumption}\label{asm1}
  (i) The action sets $\mathcal{X}_1$ and $\mathcal{X}_2$ are compact and convex.\\
  (ii) For each agent $v_1^i\in\mathcal{V}_1$, the cost $f_{1,i}(\cdot,\cdot)$ is convex-concave over $\mathcal{X}_1\times\mathcal{X}_2$. Similarly, for each agent $v_2^i\in\mathcal{V}_2$, the cost $f_{2,i}(\cdot,\cdot)$ is concave-convex over $\mathcal{X}_1\times\mathcal{X}_2$.\\
  (iii) For each $v_1^i\in\mathcal{V}_1$, $f_{1,i}(x_1,x_2)$ is $L_{1,1}$-Lipschitz continuous in $x_1\in\mathcal{X}_1$ for any $x_2\in\mathcal{X}_2$ and $L_{1,2}$-Lipschitz continuous in $x_2\in\mathcal{X}_2$ for any $x_1\in\mathcal{X}_1$, i.e.,
  \begin{align*}
    |f_{1,i}(x_1,x_2) - f_{1,i}(x_1',x_2)|&\le L_{1,1}\|x_1 - x_1'\|,\ \forall x_2\in\mathcal{X}_2,\\
  |f_{1,i}(x_1,x_2) - f_{1,i}(x_1,x_2')|&\le L_{1,2}\|x_2 - x_2'\|,\ \forall x_1\in\mathcal{X}_1.
\end{align*}
  Similarly, for each $v_2^i\in\mathcal{V}_2$, $f_{2,i}(x_1,x_2)$ is $L_{2,1}$-Lipschitz continuous in $x_2\in\mathcal{X}_2$ for any $x_1\in\mathcal{X}_1$ and $L_{2,2}$-Lipschitz continuous in $x_1\in\mathcal{X}_1$ for any $x_2\in\mathcal{X}_2$.
\end{assumption}
Assumption \ref{asm1} ensures the existence of a NE, see \cite[Theorem 2.5]{gharesifard2013distributed}.
\begin{lem}
  Under Assumption \ref{asm1}, there exists a NE for the considered subnetwork zero-sum game.
\end{lem}
In the two subnetworks, each agent only knows its own cost function. Within each subnetwork, agents can exchange information with their neighbors, i.e., $v_l^j$ can pass information to $v_l^i$ at time $k$ if $v_l^j\in\mathcal{N}_l^i(k)$. Moreover, each subnetwork can also obtain information about the other subnetwork via $\mathcal{G}_{12}(k)$, i.e., $v_{3-l}^j$ can pass the information of $\Sigma_{3-l}$ to $v_l^i$ at time $k$ if $v_{3-l}^j\in\mathcal{N}_{12,l}^i(k)$. For brevity, we use $i\in\mathcal{V}_l$ to represent the agent $v_l^i$ later when there is no confusion. We make the following assumption on the communication networks.
\begin{assumption}\label{asm2}
   For $l = 1,2$, the graph sequence $\mathcal{G}_l(k)$ is uniformly jointly strongly connected (i.e., there exists an positive integer $B_l$ such that $\bigcup_{t=k}^{k + B_l - 1}\mathcal{G}_l(t)$ is strongly connected for $k\ge 0$) and every agent in $\Sigma_l$ has at least one neighbor in $\Sigma_{3-l}$ for all $k$. Furthermore, the associated adjacency matrices $W_l(k) = [w_{l,ij}(k)]$ and $W_{12}(k) = [w_{12,ij}(k)]$ satisfy:\\
   (i) there exists a scalar $\eta\in(0,1)$ such that $w_{l,ij}(k)\ge\eta$ when $j\in\mathcal{N}_l^i(k)$, and $w_{l,ij}(t) = 0$ otherwise; the same for $w_{12,ij}(k)$;\\
  (ii) $\sum_{j = 1}^{n_l}w_{l,ij}(k) = \sum_{i = 1}^{n_l}w_{l,ij}(k) = 1$;\\
  (iii) $\sum_{j = 1}^{n_{3-l}}w_{12,ij}(k) = 1$, $i\in\mathcal{V}_l$.
\end{assumption}
\begin{remark}
The weight rules (i), (ii) and (iii) have been widely used in distrbuted optimization \cite{srivastava2013distributed,nedic2010constrained} and distributed NE seeking \cite{Koshal,lou2015nash}. In addition, the assumption that every agent in $\Sigma_l$ has at least one neighbor in $\Sigma_{3-l}$ for all $k$ is required to ensure that every agent in $\Sigma_l$ can receive the information about $\Sigma_{3-l}$ at every time.
\end{remark}
Suppose that at every time $t$, agent $i$ in $\Sigma_l$ ($l\in\{1,2\}$) maintains an estimate of its subnetwork's action as $x_{l,i}(t)$ and receives the information about its adversarial subnetwork. For clarity, we mainly focus on subnetwork $\Sigma_1$ in the subsequent formulation. At time $t$, each agent $i$ in $\Sigma_1$ receives the information about $\Sigma_2$ from the agents $j\in\mathcal{N}_{12,i}^1(t)$ and forms an estimation for $\Sigma_2$'s action, which is denoted by $u_{2,i}(t) \triangleq \sum_{j\in\mathcal{N}_{12,i}^1(t)}w_{12,ij}(t)x_{2,j}(t)$. Then agent $i$ obtains a cost $f_{1,i}(x_{1,i}(t),u_{2,i}(t))$. After $T$ time steps, the regret of $\Sigma_1$ associated with agent $i$ is defined as
\begin{equation}\label{regret_def}
  R_1^{(i)}(T) = \sum_{t=1}^Tf_1(x_{1,i}(t),u_{2,i}(t)) - \min_{x_1\in\mathcal{X}_1}\sum_{t=1}^Tf_1(x_1,u_{2,i}(t)),
\end{equation}
i.e., the maximum gain $\Sigma_1$ could have achieved by playing the single best fixed action in case the estimated sequence of $\Sigma_2$'s actions $\{u_{2,i}(t)\}_{t=1}^T$ and the cost functions were known in hindsight. An algorithm is no-regret for $\Sigma_1$ if for all $i$, $R_1^{(i)}(T)/T\to 0$ as $T\to\infty$. It is desirable for the subnetworks to adopt a no-regret learning algorithm since no agent wants to realize that the action sequence employed is strictly inferior to taking a fixed action at all iterations. The goal of this paper is to design a no-regret distributed learning algorithm that converges to a NE.
\begin{remark}
  The most relevant results in this context are \cite{srivastava2013distributed} and \cite{talebi2019distributed}. \cite{srivastava2013distributed} proposed a distributed Bregman-distance algorithm for saddle-point problems while we consider a game between two subnetworks. \cite{talebi2019distributed} developed a team-based dual averaging algorithm for a game (not necessarily zero-sum) between two teams over a network and proved the convergence of cost values by introducing cross-monotonicity. In comparison, we provide the regret bound of the algorithm and prove the convergence of the action profiles to a NE.
\end{remark}
\subsection{Distributed Mirror Descent}
For $l = 1,2$, let $\psi_l$ be a continuously differentiable $\sigma_l$-strongly convex function on $\mathcal{X}_l$, which means that,
\[\psi_l(y) \ge \psi_l(x) + \langle \nabla\psi_l(x), y - x\rangle + \frac{\sigma_l}{2}\|x - y\|^2,\ \forall x,y\in\mathcal{X}_l.\]
Recall from \cite{doan2018convergence} that the Bregman divergence associated with $\psi_l$ is defined as
\begin{equation}\label{breg_def}
D_{\psi_l}(x,y) \triangleq \psi_l(x) - \psi_l(y) - \langle\nabla\psi_l(y),x - y\rangle.
\end{equation}
Then we design our algorithm based on the mirror descent algorithm \cite{nemirovskij1983problem,beck2003mirror}. At time $t+1$, each agent $i\in\mathcal{V}_l$ ($l\in\{1,2\}$) receives the estimates $x_{l,j}(t)$ from $j\in\mathcal{N}_l^i(t)$ and the estimates $x_{3-l,j}(t)$ from $j\in\mathcal{N}_{12,l}^i(t)$. Let $v_{l,i}(t)$ and $u_{3-l,i}(t)$ be the weighted average of the estimates from $\Sigma_l$ and $\Sigma_{3-l}$, respectively. Then for $l\in\{1,2\}$, each agent $i\in\mathcal{V}_l$ evaluates the subgradients of the local costs $f_{l,i}$ at $v_{l,i}(t)$by
\begin{equation}\label{sub_gradient}
  \begin{split}
    &g_{1,i}(t)\in\partial_1f_{1,i}(v_{1,i}(t),u_{2,i}(t)),\\
    &g_{2,i}(t)\in\partial_2f_{2,i}(u_{1,i}(t),v_{2,i}(t)),
  \end{split}
\end{equation}
and updates its estimate by the mirror descent scheme \eqref{update_1}. We summarize the procedures in Algorithm \ref{alg1}.
\begin{algorithm}[H]\caption{Distributed mirror descent algorithm}\label{alg1}
{\it Initialize:} For $l\in\{1,2\}$: let $x_{l,i}(0)  \in \mathcal{X}_l$ for each $i \in\mathcal{V}_l$.

{\it Iterate until $t\ge T$:}

\quad{\bf Communication and distributed averaging:} For $l\in\{1,2\}$,
\begin{align}
   v_{l,i}(t) &:= \sum_{j\in\mathcal{N}_{l}^{i}(t)}w_{l,ij}(t)x_{l,j}(t),\quad i\in\mathcal{V}_l\label{commu_v}\\
   u_{l,i}(t) &:= \sum_{j\in\mathcal{N}_{12,3-l}^{i}(t)}w_{12,ij}(t)x_{l,j}(t),\quad i\in\mathcal{V}_{3-l}\label{commu_u}
 \end{align}
\quad{\bf Update of $x_{l,i}(t)$:} For $l\in\{1,2\}$, $i\in\mathcal{V}_l$,\\
\indent $\qquad$$\qquad$receive the subgradients $g_{l,i}(t)$ based on \eqref{sub_gradient}\\
\indent $\qquad$$\qquad$update the estimates $x_{l,i}(t+1)$ by
\begin{align}
x_{l,i}(t+1) &= \arg\min_{x_l\in\mathcal{X}_l}\LARGE\{\langle g_{l,i}(t), x_l - v_{l,i}(t)\rangle\notag\\
&\qquad\quad + \frac{1}{\alpha(t)}D_{\psi_l}(x_l,v_{l,i}(t))\LARGE\},\label{update_1}
\end{align}
where $\{\alpha(t)\}_{t=0}^{\infty}$ is a positive non-increasing sequence.
\end{algorithm}

By Assumption \ref{asm1}(iii) and the definition of the dual norm, we have
\begin{equation}\label{bound_subgradient}
  \|g_{1,i}(t)\|_{\ast}\le L_{1,1},\ \|g_{2,i}(t)\|_{\ast}\le L_{2,1},
\end{equation}
which was frequently used in the convergence analysis of distributed algorithms \cite{lou2015nash,nedic2010constrained}. For the subsequent analysis, we make the following assumption.
\begin{assumption}\label{asm3}
  For $l = 1,2$, the Bregman divergence $D_{\psi_l}(x,y)$ is convex in $y$ and satisfies
  \begin{equation}\label{reciprocity}
    x_k\to x\quad\Rightarrow \quad D_{\psi_l}(x_k,x)\to 0.
  \end{equation}
\end{assumption}
\begin{remark}
  $D_{\psi_l}(\cdot,y)$ is always strictly convex and the assumption about the convexity in $y$ has been widely used in distributed optimization (see e.g., \cite{doan2018convergence,li2016distributed}). The requirement \eqref{reciprocity} in Assumption \ref{asm3} is called the reciprocity of the Bregman divergence, which has also been used in \cite{mertikopoulos2019learning,bravo2018bandit}. In particular, if $\psi_1 = \psi_2 = \frac{1}{2}\|x\|_2^2$, Algorithm \ref{alg1} degenerates to the projected subgradient dynamics in \cite{lou2015nash} and Assumption \ref{asm3} is naturally satisfied.
\end{remark}

\section{REGRET ANALYSIS}
We begin this section by establishing preliminary lemmas and then provide an upper bound on each agent's regret.
\subsection{Preliminary Analysis}
First, we state a basic result about the Bregman divergence.
\begin{lem}\label{lem2}
  Let $\psi$ be a continuously differentiable $\sigma$-strongly convex function on $\mathcal{X}$. Then the Bregman divergence defined by \eqref{breg_def} satisfies
  \begin{align}
    D_{\psi}(y,x) - D_{\psi}(y,z) - D_{\psi}(z,x) &= \langle\nabla\psi(z) - \nabla\psi(x),y-z\rangle,\label{breg_prop1}\\
    D_{\psi}(x,y)&\ge \frac{\sigma}{2}\|x - y\|^2,\label{breg_prop2}
  \end{align}
  for all $x,y,z\in\mathcal{X}$.
\end{lem}
Lemma \ref{lem2} is widely used in the analysis of mirror descent algorithms \cite{doan2018convergence,shahrampour2017distributed} and can be easily obtained from \eqref{breg_def}. We next state a result from distributed optimization \cite{nedic2010constrained}.
\begin{lem}\label{lem3}
  Let $\Phi_l(t,s) = W_l(t)W_l(t-1)\cdots W_l(s)$ ($l = 1,2$) be the transition matrices for $\Sigma_l$ ($l = 1,2$). Suppose that Assumption \ref{asm2} holds. Then for all $t,s$ with $t\ge s\ge 0$, we have
  \[\bigg|[\Phi_l(t,s)]_{ij} - \frac{1}{n_l}\bigg| \le \Gamma_l\theta_l^{t-s},\quad l = 1,2\]
  where $\Gamma_l = (1 - \eta/4n_l^2)^{-2}$ and $\theta_l = (1 - \eta/4n_l^2)^{1/B_l}$.
\end{lem}
Based on \eqref{update_1}, we can establish an error bound between the weighted estimate $v_{l,i}(t)$ and the new estimate $x_{l,i}(t+1)$.
\begin{lem}\label{lem4}
  Let Assumption \ref{asm1} hold. Suppose that $v_{l,i}(t)$ and $x_{l,i}(t+1)$ are generated by Algorithm \ref{alg1}. Then for each $l\in\{1,2\}$ and $i\in\mathcal{V}_l$, we have
  \begin{equation}\label{lem_bound_1}
    \|v_{l,i}(t) - x_{l,i}(t+1)\| \le \frac{\alpha(t)}{\sigma_l}\|g_{l,i}(t)\|_{\ast} \le L_{l,1}\frac{\alpha(t)}{\sigma_l}.
  \end{equation}
\end{lem}
Lemma \ref{lem4} is a consequence of the optimality condition. See Appendix for its proof. Furthermore, let $\bar{x}_l(t) = \frac{1}{n_l}\sum_{i=1}^{n_l}x_{l,i}(t)$ be the average state of $\Sigma_l$ at time $t$, and then we obtain the following error bounds based on Lemma \ref{lem3} and Lemma \ref{lem4}.
\begin{lem}\label{lem5}
  Let Assumptions \ref{asm1}-\ref{asm2} hold. Suppose that $x_{l,i}(t)$, $v_{l,i}(t)$, and $u_{l,i}(t)$ are generated by Algorithm \ref{alg1}. Then, for each $l\in\{1,2\}$ and $i\in\mathcal{V}_l$, for any $t$,
  \begin{align}
    \|x_{l,i}(t) - \bar{x}_{l}(t)\| &\le H_l(t),\label{lem_bound_2}\\
    \|\bar{x}_{l}(t) - v_{l,i}(t)\| &\le H_l(t),\label{lem_bound_3}\\
    \|\bar{x}_{l}(t) - u_{l,i}(t)\| &\le H_l(t),\label{lem_bound_4}
  \end{align}
  where
\begin{align*}
  H_l(t) &= n_l\Gamma_l\theta_l^{t-1}\Lambda_l + \frac{2}{\sigma_l}L_{l,1}\alpha(t-1)\\
  &\quad + \frac{1}{\sigma_l}\left(n_lL_{l,1}\Gamma_l\sum_{s=1}^{t-1}\theta_l^{t-1-s}\alpha(s-1)\right),
\end{align*}
and $\Lambda_l\triangleq \max_{i\in\mathcal{V}_l}\|x_{l,i}(0)\|$.
\end{lem}
In Lemma \ref{lem5}, \eqref{lem_bound_2} is a fundamental result in the analysis of distributed mirror descent algorithms \cite{doan2018convergence,li2016distributed} and \eqref{lem_bound_3}-\eqref{lem_bound_4} can be easily established based on \eqref{lem_bound_2}. For completeness, we give its proof in Appendix. Finally, we establish the following result similar to that of \cite{shahrampour2017distributed,li2016distributed} for distributed optimization.
\begin{lem}\label{lem6}
  Let Assumptions \ref{asm1}-\ref{asm3} hold. Suppose that $x_{l,i}(t)$ and $g_{l,i}(t)$ are generated by Algorithm \ref{alg1}. Then, for $l\in\{1,2\}$, $i\in\mathcal{V}_l$, for any $\breve{x}_l\in\mathcal{X}_l$, there exists $\Upsilon_l\triangleq\max\{D_{\psi_l}(\breve{x}_l, x_l)\mid \forall x_l\in\mathcal{X}_l\}$ such that
  \begin{equation}\label{lem_bound_5}
    \frac{1}{n_l}\sum_{t=1}^T\sum_{i=1}^{n_l}\langle g_{l,i}(t), x_{l,i}(t+1) - \breve{x}_l\rangle \le \frac{\Upsilon_l^2}{\alpha(T)}.
  \end{equation}
\end{lem}
By the optimality condition and \eqref{breg_prop1}, we obtain \eqref{lem_bound_5}. The detailed proof is provided in Appendix.
\subsection{Regret Bound}
Equipped with the above lemmas, we are ready to provide a regret bound of Algorithm 1.
\begin{thm}\label{thm1}
  Under Assumptions \ref{asm1}-\ref{asm3}, the regret defined by \eqref{regret_def} can be bounded as
  \begin{align}
    R_1^{(i)}(T)&\le \frac{4}{\sigma_1}\left(\frac{n_1L_{1,1}^2\Gamma_1}{1 - \theta_1} + 2L_{1,1}^2\right)\sum_{t=1}^T\alpha(t-1) + \frac{\Upsilon_l^2}{\alpha(T)}\notag\\
    &\quad + \frac{12}{\sigma_2}\left(\frac{n_2L_{1,2}L_{2,1}\Gamma_2}{1 - \theta_2} + 2L_{1,2}L_{2,1}\right)\sum_{t=1}^T\alpha(t-1)\notag\\
    &\quad + 4\sum_{t=1}^T\sum_{l=1}^2n_l\Gamma_l\theta_l^{t-1}\Lambda_l + \sum_{t=1}^T\frac{\alpha(t)}{\sigma_1}L_{1,1}^2.\label{regret_bound_final}
  \end{align}
\end{thm}
\begin{proof}
Denote by $\bar{x}_1(t) = \frac{1}{n_1}\sum_{i=1}^{n_1}x_{1,i}(t)$ and $x_1^{\circ} = \arg\min_{x_1\in\mathcal{X}_1}\sum_{t=1}^Tf_1(x_1,u_{2,i}(t))$. Recalling from Assumption \ref{asm1}(iii) that $f_{1,i}(x_1,x_2)$ is $L_{1,1}$-Lipschitz continuous in $x_1\in\mathcal{X}_1$ for any $x_2\in\mathcal{X}_2$, we have
\begin{align}
  R_1^{(i)}(T)&=\sum_{t=1}^T\left(f_1(x_{1,i}(t),u_{2,i}(t)) - f_1(x_1^{\circ},u_{2,i}(t))\right)\notag\\
  &= \sum_{t=1}^T\Big(f_1(x_{1,i}(t),u_{2,i}(t)) - f_1(\bar{x}_1(t),u_{2,i}(t))\notag\\
  &\quad + f_1(\bar{x}_1(t),u_{2,i}(t)) - f_1(x_1^{\circ},u_{2,i}(t))\Big)\notag\\
  &\le \sum_{t=1}^T\left(\underbrace{f_1(\bar{x}_1(t),u_{2,i}(t)) - f_1(x_1^{\circ},u_{2,i}(t))}_{A_{1,i}(t)}\right)\notag\\
  &\quad + L_{1,1}\sum_{t=1}^T\|x_{1,i}(t) - \bar{x}_1(t)\|,\label{regret_2}
\end{align}
Similarly, by Assumption \ref{asm1}(iii),
\begin{align}
  A_{1,i}(t) &= f_1(\bar{x}_1(t),u_{2,i}(t)) - f_1(\bar{x}_1(t),\bar{x}_2(t))\notag\\
  &\quad + f_1(\bar{x}_1(t),\bar{x}_2(t)) - f_1(x_1^{\circ},\bar{x}_2(t))\notag\\
  &\quad + f_1(x_1^{\circ},\bar{x}_2(t)) - f_1(x_1^{\circ},u_{2,i}(t))\notag\\
  &\le 2L_{1,2}\|u_{2,i}(t) - \bar{x}_2(t)\|\notag\\
  &\quad + \underbrace{f_1(\bar{x}_1(t),\bar{x}_2(t)) - f_1(x_1^{\circ},\bar{x}_2(t))}_{B_1(t)}.\label{A_1}
\end{align}
As a result,
\begin{align}
  R_1^{(i)}(T) &\le \sum_{t=1}^T\left(2L_{1,2}\|u_{2,i}(t) - \bar{x}_2(t)\| + L_{1,1}\|x_{1,i}(t) - \bar{x}_1(t)\|\right)\notag\\
  &\quad + \sum_{t=1}^TB_1(t).\label{regret_2_1}
\end{align}
Furthermore, according to the convexity of $f_{1,i}(\cdot,x_2)$ (Assumption \ref{asm1}(ii)) and Assumption \ref{asm1}(iii),
\begin{align}
  B_1(t) &\overset{\eqref{cost_def}}{=} \frac{1}{n_1}\sum_{i=1}^{n_1}[f_{1,i}(\bar{x}_1(t),\bar{x}_2(t)) - f_{1,i}(x_{1,i}(t),\bar{x}_2(t))\notag\\
  &\quad + f_{1,i}(x_{1,i}(t),\bar{x}_2(t)) - f_{1,i}(v_{1,i}(t),\bar{x}_2(t))\notag\\
  &\quad + f_{1,i}(v_{1,i}(t),\bar{x}_{2}(t)) - f_{1,i}(v_{1,i}(t),u_{2,i}(t))\notag\\
  &\quad + f_{1,i}(x_1^{\circ},u_{2,i}(t)) - f_{1,i}(x_1^{\circ},\bar{x}_2(t))\notag\\
  &\quad + f_{1,i}(v_{1,i}(t),u_{2,i}(t)) - f_{1,i}(x_1^{\circ},u_{2,i}(t))]\notag\\
  &\le \frac{1}{n_1}\sum_{i=1}^{n_1}\Big[L_{1,1}(\|x_{1,i}(t) - \bar{x}_1(t)\| + \|x_{1,i}(t) - v_{1,i}(t)\|)\notag\\
  &\quad + 2L_{1,2}\|u_{2,i}(t) - \bar{x}_2(t)\| + \underbrace{\langle g_{1,i}(t),v_{1,i}(t) - x_1^{\circ}\rangle}_{C_{1,i}(t)}\Big].\label{B_1}
\end{align}
By the Cauchy-Schwarz inequality, we obtain
\begin{align}
  C_{1,i}(t) &= \langle g_{1,i}(t),v_{1,i}(t) - x_{1,i}(t+1)\rangle\notag\\
  &\quad + \langle g_{1,i}(t), x_{1,i}(t+1) - x_1^{\circ}\rangle\notag\\
  &\le \|g_{1,i}(t)\|_{\ast}\|v_{1,i}(t) - x_{1,i}(t+1)\|\notag\\
  &\quad + \langle g_{1,i}(t), x_{1,i}(t+1) - x_1^{\star}\rangle\notag\\
  &\le L_{1,1}^2\frac{\alpha(t)}{\sigma_1} + \langle g_{1,i}(t), x_{1,i}(t+1) - x_1^{\circ}\rangle,\label{C_1}
\end{align}
where the last inequality follows from Lemma \ref{lem4} and \eqref{bound_subgradient}. Consequently,
\begin{align}
  R_1^{(i)}(T) &\le \sum_{t=1}^T\left(2L_{1,2}\|u_{2,i}(t) - \bar{x}_2(t)\| + L_{1,1}\|x_{1,i}(t) - \bar{x}_1(t)\|\right)\notag\\
  &\quad + \sum_{t=1}^T\frac{1}{n_1}\sum_{i=1}^{n_1}\Big[L_{1,1}(\|x_{1,i}(t) - \bar{x}_1(t)\|\notag\\
  &\qquad + \|x_{1,i}(t) - v_{1,i}(t)\|) + 2L_{1,2}\|u_{2,i}(t) - \bar{x}_2(t)\|\Big]\notag\\
  &\quad + \sum_{t=1}^T\frac{1}{n_1}\sum_{i=1}^{n_1}\langle g_{1,i}(t), x_{1,i}(t+1) - x_1^{\circ}\rangle\notag\\
  &\quad+ \sum_{t=1}^T\frac{\alpha(t)}{\sigma_1}L_{1,1}^2.\label{regret_2_2}
\end{align}
Applying \eqref{lem_bound_2} and \eqref{lem_bound_3} to the triangle inequality of the norm yields
\begin{equation}\label{lem_bound_3_new}
  \|x_{l,i}(t) - v_{l,i}(t)\| \le 2H_l(t).
\end{equation}
Substituting \eqref{lem_bound_2}, \eqref{lem_bound_4}, \eqref{lem_bound_5} and \eqref{lem_bound_3_new} into \eqref{regret_2_2}, we obtain
\begin{align}
  &\quad R_1^{(i)}(T)\notag\\
  &\le \frac{4}{\sigma_1}\sum_{t=1}^T\left(n_1L_{1,1}^2\Gamma_1\sum_{s=1}^{t-1}\theta_1^{t-1-s}\alpha(s-1) + 2L_{1,1}^2\alpha(t-1)\right)\notag\\
  &\quad + \sum_{t=1}^T\frac{\alpha(t)}{\sigma_1}L_{1,1}^2 + \frac{8}{\sigma_2}L_{1,2}L_{2,1}\sum_{t=1}^T\alpha(t-1) \notag\\
  &\quad + \frac{4}{\sigma_2}\sum_{t=1}^T\left(n_2L_{1,2}L_{2,1}\Gamma_2\sum_{s=1}^{t-1}\theta_2^{t-1-s}\alpha(s-1)\right)\notag\\
  &\quad + 4\sum_{t=1}^T\sum_{l=1}^2n_l\Gamma_l\theta_l^{t-1}\Lambda_l + \frac{\Upsilon_l^2}{\alpha(T)}.\label{regret_bound_2}
\end{align}
By exchanging the order of summation, for $l\in\{1,2\}$,
\begin{align}
&\quad\sum_{t=1}^T\sum_{s=1}^{t-1}\theta_l^{t-1-s}\alpha(s-1)\notag\\
&\le \sum_{t=1}^T\sum_{s=0}^{T-1}\theta_l^s\alpha(t-1)\le \frac{1}{1 - \theta_l}\sum_{t=1}^T\alpha(t-1).\label{exchange_sum}
\end{align}
 This combined with \eqref{regret_bound_2} produces \eqref{regret_bound_final}.
\end{proof}
Theorem \ref{thm1} provides an upper bound on the individual regret for each agent in subnetwork $\Sigma_1$ in the case of a general step-size sequence $\{\alpha(t)\}_{t=1}^T$. Note that the impact of the communication network is incorporated in the constants $\Gamma_l$ and $\theta_l$. Moreover, a regret bound for each agent in $\Sigma_2$ can be similarly established. Next, we characterize the regret bound under two specific step-size sequences.
\begin{Col}\label{col1}
  Under the same conditions stated in Theorem \ref{thm1}, Algorithm \ref{alg1} with a constant step-size $\alpha(t) \equiv 1/\sqrt{T}$ yields the regret bound of order
  \[R_1^{(i)}(T)\le O(\sqrt{T}).\]
  In addition, Algorithm \ref{alg1} with a diminishing step-size sequence $\alpha(t) = t^{-(\frac{1}{2} + \epsilon)}$ ($\epsilon\in(0,\frac{1}{2})$) yields the regret bound of order
  \[R_1^{(i)}(T) \le O(T^{\frac{1}{2} + \epsilon}).\]
\end{Col}
This corollary follows by $\sum_{t=1}^T\alpha(t) \le 1 + \int_{t=1}^Tt^{-(\frac{1}{2} + \epsilon)} \le T^{\frac{1}{2} - \epsilon}$ and $1/\alpha(T) = T^{\frac{1}{2} + \epsilon}$. Furthermore, Corollary \ref{col1} shows that both constant and diminishing step-sizes can make Algorithm \ref{alg1} no-regret, and the regret rate has the same order as that of distributed online optimization \cite{shahrampour2017distributed}.

\section{CONVERGENCE ANALYSIS}
In this section, we study the convergence properties of Algorithm \ref{alg1} under both diminishing and constant step-sizes.
\subsection{Diminishing Step-Size}
In this subsection, we adopt a diminishing step-size sequence in Algorithm \ref{alg1} and make the following assumption.
\begin{assumption}\label{asm4}
  $\sum_{t=0}^{\infty}\alpha(t) = \infty$ and $\sum_{t=0}^{\infty}\alpha^2(t) < \infty$.
\end{assumption}
To facilitate the convergence analysis, we first state a well-known result about non-negative sequences \cite{polyak1987introduction}.
\begin{lem}\label{lem9}
  Let $\{a_t\}$, $\{b_t\}$ and $\{c_t\}$ be non-negative sequences with $\sum_{t=0}^{\infty}b_t < \infty$. If $a_{t+1} \le a_t + b_t - c_t$ holds for any $t$, then $a_t$ converges to a finite number and $\sum_{t=0}^Tc_t < \infty$.
\end{lem}
Next, we provide a convergence result on the actual sequence of actions; namely, each paired sequence $\{(x_{1,i}(t),x_{2,j}(t))\}$ converges to the NE.
\begin{thm}\label{thm2}
  Suppose that Assumptions \ref{asm1}-\ref{asm4} hold and cost function $U$ is strictly convex-concave. Then Algorithm \ref{alg1} generates a sequence that converges to the unique NE $x^{\ast} = (x_1^{\ast},x_2^{\ast})$, i.e.,
  \begin{equation}\label{convergence_result}
    \lim_{t\to\infty}x_{1,i}(t) = x_1^{\ast},\quad \lim_{t\to\infty}x_{2,j}(t) = x_2^{\ast},\ \forall i\in\mathcal{V}_1, j\in\mathcal{V}_2.
  \end{equation}
\end{thm}
\begin{proof}
By the convexity of $f_{1,i}$ with respect to $x_1$ and recalling that $g_{1,i}(t)\in\partial_1 f_{1,i}(v_{1,i}(t),u_{2,i}(t))$, for all $x_1\in\mathcal{X}_1$, we have
\begin{align}
  &\quad\langle g_{1,i}(t), x_1 - v_{1,i}(t)\rangle\notag\\
   &\le f_{1,i}(x_1, u_{2,i}(t)) - f_{1,i}(v_{1,i}(t),u_{2,i}(t))\notag\\
  &= f_{1,i}(x_1, u_{2,i}(t)) - f_{1,i}(x_1, \bar{x}_2(t))\notag\\
  &\quad + f_{1,i}(x_1, \bar{x}_2(t)) - f_{1,i}(\bar{x}_1(t),\bar{x}_2(t))\notag\\
  &\quad + f_{1,i}(\bar{x}_1(t),\bar{x}_2(t)) - f_{1,i}(v_{1,i}(t),\bar{x}_2(t))\notag\\
  &\quad + f_{1,i}(v_{1,i}(t),\bar{x}_2(t)) - f_{1,i}(v_{1,i}(t),u_{2,i}(t))\notag\\
  &\le L(\|v_{1,i}(t) - \bar{x}_1(t)\| + 2\|u_{2,i}(t) - \bar{x}_2(t)\|)\notag\\
  &\quad + f_{1,i}(x_1, \bar{x}_2(t)) - f_{1,i}(\bar{x}_1(t),\bar{x}_2(t))\label{inner_part_1},
\end{align}
where $L = \max\{L_{1,1},L_{1,2},L_{2,1},L_{2,2}\}$ and the last inequality follows from Assumption \ref{asm1}(iii). Moreover, Young's inequality ($2\langle a,b\rangle\le c\|a\|_{\ast}^2 + \frac{1}{c}\|b\|^2$) and $\|g_{1,i}(t)\|_{\ast}\le L$ together derive
\begin{align}\label{inner_part_2}
  &\quad\langle\alpha(t)g_{1,i}(t), v_{1,i}(t) - x_{1,i}(t+1)\rangle\notag\\
   &\le \frac{\alpha^2(t)L^2}{2\sigma_1} + \frac{\sigma_1}{2}\|v_{1,i}(t) - x_{1,i}(t+1)\|^2.
\end{align}
Then, with $U(\cdot,\cdot)\triangleq f_1(\cdot,\cdot)$, we obtain
\begin{align}
  &\quad\sum_{i=1}^{n_1}\langle \alpha(t)g_{1,i}(t),x_1 - x_{1,i}(t+1)\rangle\notag\\
  &\le \alpha(t)\sum_{i=1}^{n_1}(f_{1,i}(x_1, \bar{x}_2(t)) - f_{1,i}(\bar{x}_1(t),\bar{x}_2(t)))\notag\\
  &\quad + \alpha^2(t)\frac{n_1L^2}{2\sigma_1} + \frac{\sigma_1}{2}\sum_{i=1}^{n_1}\|v_{1,i}(t) - x_{1,i}(t+1)\|^2\notag\\
  &\quad + \alpha(t)L\sum_{i=1}^{n_1}e_{1,i}(t)\notag\\
  &= n_1\alpha(t)(U(x_1,\bar{x}_2(t)) - U(\bar{x}_1(t),\bar{x}_2(t))) + \alpha(t)L\sum_{i=1}^{n_1}e_{1,i}(t)\notag\\
  &\quad + \alpha^2(t)\frac{n_1L^2}{2\sigma_1} + \frac{\sigma_1}{2}\sum_{i=1}^{n_1}\|v_{1,i}(t) - x_{1,i}(t+1)\|^2,\label{inner_part}
\end{align}
where $e_{1,i}(t) = \|v_{1,i}(t) - \bar{x}_1(t)\| + 2\|u_{2,i}(t) - \bar{x}_2(t)\|$. On the other hand, from Lemma \ref{lem2}, we get
\begin{align}
  &\quad\langle\nabla\psi_1(x_{1,i}(t+1)) - \nabla\psi_1(v_{1,i}(t)), x_1 - x_{1,i}(t+1)\rangle\notag\\
  &= D_{\psi_1}(x_1,v_{1,i}(t)) - D_{\psi_1}(x_1, x_{1,i}(t+1))\notag\\
  &\quad - D_{\psi_1}(x_{1,i}(t+1),v_{1,i}(t)),\label{three_point}
\end{align}
and
\begin{equation}\label{strong_conv}
  D_{\psi_1}(x_{1,i}(t+1),v_{1,i}(t)) \ge \frac{\sigma_1}{2}\|x_{1,i}(t+1) - v_{1,i}(t)\|^2.
\end{equation}
Meanwhile, using Assumption \ref{asm3} and Jensen's inequality,
\begin{align}
  D_{\psi_1}(x_1,v_{1,i}(t)) &= D_{\psi_1}\left(x_1,\sum_{j=1}^{n_1}w_{1,ij}x_{1,j}(t)\right)\notag\\
  &\le \sum_{j=1}^{n_1}w_{1,ij}(t)D_{\psi_1}(x_1,x_{1,j}(t)).\label{conv}
\end{align}
Plugging \eqref{strong_conv} and \eqref{conv} back in \eqref{three_point}, we derive
\begin{align}
  &\quad\sum_{i=1}^{n_1}\langle\nabla\psi_1(x_{1,i}(t+1)) - \nabla\psi_1(v_{1,i}(t)), x_1 - x_{1,i}(t+1)\rangle\notag\\
  &\le \sum_{i=1}^{n_1}\Big(D_{\psi_1}(x_1,x_{1,i}(t)) - D_{\psi_1}(x_1, x_{1,i}(t+1))\Big)\notag\\
  &\quad - \frac{\sigma_1}{2}\sum_{i=1}^{n_1}\|v_{1,i}(t) - x_{1,i}(t+1)\|^2,\label{dist_part}
\end{align}
because $\sum_{i=1}^{n_1}w_{1,ij}(t) = 1$. From \eqref{update_1} and the optimality condition, for all $x_1\in\mathcal{X}_1$,
\begin{align}
  &\Big\langle \nabla \psi_1(x_{1,i}(t+1)) - \nabla \psi_1(v_{1,i}(t)) + \alpha(t)g_{1,i}(t),\notag\\
  &\ \  x_1 - x_{1,i}(t+1)\Big\rangle \ge 0.\label{opt_cond}
\end{align}
Summing up \eqref{inner_part} and \eqref{dist_part}, with \eqref{opt_cond}, we obtain
\begin{align*}
  0 &\le \sum_{i=1}^{n_1}(D_{\psi_1}(x_1,x_{1,i}(t)) - D_{\psi_1}(x_1, x_{1,i}(t+1))) \\
  &\quad + n_1\alpha(t)(U(x_1,\bar{x}_2(t)) - U(\bar{x}_1(t),\bar{x}_2(t)))\\
  &\quad + \alpha^2(t)\frac{n_1L^2}{2\sigma_1} + \alpha(t)L\sum_{i=1}^{n_1}e_{1,i}(t).
\end{align*}
Rearranging the terms yields
\begin{align}
  &\quad\sum_{i=1}^{n_1}D_{\psi_1}(x_1, x_{1,i}(t+1))\notag\\ &\le \sum_{i=1}^{n_1}D_{\psi_1}(x_1, x_{1,i}(t)) + \alpha(t)L\sum_{i=1}^{n_1}e_{1,i}(t) + \alpha^2(t)\frac{n_1L^2}{2\sigma_1} \notag\\
  &\quad + n_1\alpha(t)(U(x_1,\bar{x}_2(t)) - U(\bar{x}_1(t),\bar{x}_2(t))).\label{network_1}
\end{align}
Similarly, for subnetwork $\Sigma_2$, we get the relation
\begin{align}
  &\quad\sum_{i=1}^{n_2}D_{\psi_2}(x_2, x_{2,i}(t+1))\notag\\ &\le \sum_{i=1}^{n_2}D_{\psi_2}(x_2, x_{2,i}(t)) + \alpha(t)L\sum_{i=1}^{n_2}e_{2,i}(t) + \alpha^2(t)\frac{n_2L^2}{2\sigma_2} \notag\\
  &\quad + n_2\alpha(t)(U(\bar{x}_1(t),\bar{x}_2(t)) - U(\bar{x}_1(t), x_2)),\label{network_2}
\end{align}
where $e_{2,i}(t) = \|v_{2,i}(t) - \bar{x}_2(t)\| + 2\|u_{1,i}(t) - \bar{x}_1(t)\|$. Let $x^{\ast} = (x_1^{\ast},x_2^{\ast})$ be the NE. Consider the Lyapunov function
\[V(t,x_1^{\ast},x_2^{\ast}) \triangleq \frac{1}{n_1}\sum_{i=1}^{n_1}D_{\psi_1}(x_1^{\ast},x_{1,i}(t)) + \frac{1}{n_2}\sum_{i=1}^{n_2}D_{\psi}(x_2^{\ast},x_{2,i}(t)).\]
Obviously,
\begin{align}
  &\quad V(t+1,x_1^{\ast},x_2^{\ast})\notag\\
   &\le V(t,x_1^{\ast},x_2^{\ast}) - \alpha(t)(U(\bar{x}_1(t), x_2^{\ast}) - U(x_1^{\ast},\bar{x}_2(t)))\notag\\
  &\quad + \alpha(t)L\sum_{l=1}^2\frac{1}{n_1}\sum_{i=1}^{n_1}e_{l,i}(t) + \alpha^2(t)L^2\sum_{l=1}^2\frac{1}{2\sigma_l}.\label{V_iter}
\end{align}
Note by the definition of NE that
\begin{equation}\label{U_diff}
  U(\bar{x}_1(t), x_2^{\ast})\ge U(x_1^{\ast},x_2^{\ast})\ge U(x_1^{\ast},\bar{x}_2(t)),
\end{equation}
and $\sum_{t=1}^{\infty}\alpha^2(t) < \infty$. Therefore, in order to use Lemma \ref{lem9}, we only need to show
\begin{equation}\label{converge_cond}
  \sum_{t=1}^{\infty}\alpha(t)\left(\sum_{i=1}^{n_1}e_{1,i}(t) + \sum_{i=1}^{n_2}e_{2,i}(t)\right) < \infty.
\end{equation}
Applying the bounds in \eqref{lem_bound_2}-\eqref{lem_bound_4} gives
\begin{align*}
  e_{1,i}(t)
  &\le C_1\sum_{s=1}^{t-1}\theta_1^{t-1-s}\alpha(s - 1) + C_2\sum_{s=1}^{t-1}\theta_2^{t-1-s}\alpha(s - 1)\notag\\
  &\quad + C_3\alpha(t-1)+ \sum_{l=1}^22n_l\Gamma_l\theta_l^{t-1}\Lambda_l.
\end{align*}
where $C_1 = \frac{n_1L\Gamma_1}{\sigma_1}$, $C_2 = \frac{2n_2L\Gamma_2}{\sigma_2}$ and $C_3 = \frac{2L}{\sigma_1} + \frac{4L}{\sigma_2}$. Note that
\begin{align*}
  &\quad\sum_{t=1}^T\alpha(t)e_{1,i}(t)\\ &\le C_1\sum_{t=1}^T\sum_{s=1}^t\theta_1^{t-1-s}\alpha^2(s - 1) + 2\sum_{l=1}^2n_l\Gamma_l\Lambda_l\sum_{t=1}^T\alpha(t)\theta_l^{t-1}\\
  &\quad + C_2\sum_{t=1}^T\sum_{s=1}^t\theta_2^{t-1-s}\alpha^2(s - 1) + C_3\sum_{t=1}^T\alpha^2(t-1).
\end{align*}
Similar to \eqref{exchange_sum}, we have \[\sum_{t=1}^T\sum_{s=1}^{t}\theta_l^{t-1-s}\alpha^2(s-1)\le \frac{1}{1 - \theta_l}\sum_{t=1}^T\alpha^2(t-1).\]
Also, $\sum_{t=1}^T\alpha(t)\theta_1^{t-1}\le \alpha(0)\sum_{t=1}^T\theta_1^{t-1}\le\frac{\alpha(0)}{1 - \theta_1}$. Then
\begin{align}
  \sum_{t=1}^T\alpha(t)e_{1,i}(t) &\le \left(\frac{C_1}{1 - \theta_1} + \frac{C_2}{1 - \theta_2} + C_3\right)\sum_{t=1}^T\alpha^2(t-1)\notag\\
  &\quad + 2\sum_{l=1}^2\frac{n_l\Gamma_l\alpha(0)\Lambda_l}{1-\theta_l},\label{converge_cond_1}
\end{align}
implying $\sum_{t=1}^{\infty}\alpha(t)e_{1,i}(t) < \infty$. Similarly, $\sum_{t=1}^{\infty}\alpha(t)e_{2,i}(t) < \infty$, i.e., \eqref{converge_cond} holds. By Lemma \ref{lem9}, $V(t,x_1^{\ast},x_2^{\ast})$ converges to a finite number. Furthermore, by \eqref{V_iter} and \eqref{U_diff},
\[0\le\sum_{t=0}^{\infty}\alpha(t)(U(\bar{x}_1(t),x_2^{\ast}) - U(x_1^{\ast},\bar{x}_2(t))) < \infty.\]
Therefore, $\sum_{t=0}^{\infty}\alpha(t) = \infty$ yields $\lim\inf_{t\to\infty}U(\bar{x}_1(t),x_2^{\ast}) - U(x_1^{\ast},\bar{x}_2(t)) = 0$. Hence, there exists a subsequence $\{t_r\}$ such that
\[\lim_{r\to\infty}U(x_1^{\ast},\bar{x}_2(t_r)) = U(x_1^{\ast},x_2^{\ast}) = \lim_{r\to\infty}U(\bar{x}_1(t_r),x_2^{\ast}).\]
Let $(\tilde{x}_1,\tilde{x}_2)$ be a limit point of the bounded sequence $\{(\bar{x}_1(t_r),\bar{x}_2(t_r))\}$. Then there exists a subsequence $\{t_{r_p}\}$ such that $\lim_{p\to\infty}\bar{x}_l(t_{r_p}) = \tilde{x}_{l}$, and hence, by the continuity of $U(\cdot,\cdot)$,
\[U(x_1^{\ast},\tilde{x}_2) = U(\tilde{x}_1,x_2^{\ast}) = U(x_1^{\ast},x_2^{\ast}).\]
By the strict convexity-concavity of $U$, the NE is unique, i.e., $(\tilde{x}_1,\tilde{x}_2) = (x_1^{\ast},x_2^{\ast})$. Using \eqref{lem_bound_2} and \cite[Lemma 5.2]{lou2015nash}, we obtain
\begin{align*}
  \lim_{p\to\infty}x_{1,i}(t_{r_p}) &= \lim_{p\to\infty}\bar{x}_1(t_{r_p}) = x_1^{\ast},\\
  \lim_{p\to\infty}x_{2,i}(t_{r_p}) &= \lim_{p\to\infty}\bar{x}_2(t_{r_p}) = x_2^{\ast}.
\end{align*}
According to Assumption \ref{asm3}, $\lim_{p\to\infty}V(t_{r_p},x_1^{\ast},x_2^{\ast}) = 0$. Moreover, by the convergence of $V(t,x_1^{\ast},x_2^{\ast})$,
\[\lim_{t\to\infty}V(t,x_1^{\ast},x_2^{\ast}) = \lim_{p\to\infty}V(t_{r_p},x_1^{\ast},x_2^{\ast}) = 0,\]
which is incorporated with \eqref{breg_prop2} to prove \eqref{convergence_result}.
\end{proof}
\begin{remark}
   Compared with \cite{talebi2019distributed} which establishes the convergence of the running average of action iterates under the assumption of cross-monotonicity and strict monotonicity, Theorem \ref{thm2} shows that the agents' estimates of the subnetwork state can achieve consensus and converge to the NE.
 \end{remark}
Moreover, since the diminishing step-size sequence in Corollary 1 satisfies Assumption \ref{asm4}, the following conclusion is a direct corollary of Theorem \ref{thm2}.
\begin{Col}\label{col2}
	Under Assumptions \ref{asm1}-\ref{asm3}, Algorithm \ref{alg1} with a diminishing step-size sequence $\alpha(t) = t^{-(\frac{1}{2} + \epsilon)}$ ($\epsilon\in(0,\frac{1}{2})$) is a no-regret distributed updating policy while converges to the NE.
\end{Col}
\subsection{Constant Step-size}
Here we consider a constant step-size version of Algorithm \ref{alg1} (i.e., $\alpha(t) \equiv \alpha$). By Corollary \ref{col1}, we obtain a tighter regret bound under constant step-size. However, since Assumption \ref{asm4} does not hold in this situation, Theorem \ref{thm2} is no longer applicable. 
Inspired by \cite{nedic2009subgradient}, we consider the following running average actions for each node $i$ in $\Sigma_l$,
\[\hat{x}_{l,i}(t) = \frac{1}{t}\sum_{s=0}^{t-1}x_{l,i}(s),\quad\text{for}\ t\ge 1, l = 1,2.\]
Next we prove that these averages provide an approximation of NE. Specifically, let $(x_1^{\ast},x_2^{\ast})$ be a NE. Then the following theorem establishes an asymptotic error bound between the cost of the averages and the cost of NE.
\begin{thm}\label{thm3}
  Suppose Assumptions \ref{asm1}-\ref{asm3} hold. Consider Algorithm \ref{alg1} with $\alpha(t)\equiv\alpha$. Then for all $t\ge 1$, the averages $(\hat{x}_{1,i}(t),\hat{x}_{2,j}(t)), i\in\mathcal{V}_1, j\in\mathcal{V}_2$ satisfy
  \begin{align}
    &\quad |U(\hat{x}_{1,i}(t),\hat{x}_{2,j}(t)) - U(x_1^{\ast},x_2^{\ast})|\notag\\
     &\le \sum_{l=1}^2\left(\frac{\Upsilon_l^2}{t\alpha} + \frac{L^2}{\sigma_l}\alpha\right)+ 4L(K_1 + K_2)\alpha\notag\\
     &\quad + 4L\frac{1}{t}\sum_{s=0}^{t-1}\sum_{l=1}^2n_l\Gamma_l\theta_l^{s-1}\Lambda_l,
  \end{align}
  where $K_l \triangleq \frac{1}{\sigma_l}\left(\frac{n_lL\Gamma_l}{1 - \theta_l} + 2L\right)$ for $l = 1,2$.
\end{thm}
\begin{proof}
  By Assumptions \ref{asm1}(ii) and \ref{asm1}(iii), $U$ is convex and Lipschitz continuous in $x_1\in\mathcal{X}_1$ for any $x_2$. With the Jensen's inequality, we obtain
\begin{align}
  &\quad U(\hat{x}_{1,i}(t),x_2)\notag\\
   &\le \frac{1}{t}\sum_{s=0}^{t-1}U(x_{1,i}(s),x_2)\notag\\
  &= \frac{1}{t}\sum_{s=0}^{t-1}[U(\bar{x}_1(s),x_2) - U(\bar{x}_1(s),x_2) + U(x_{1,i}(s),x_2)]\notag\\
  &\le \frac{1}{t}\sum_{s=0}^{t-1}[U(\bar{x}_1(s),x_2) + L\|\bar{x}_1(s) - x_{1,i}(s)\|]\notag\\
  &= \frac{1}{t}\sum_{s=0}^{t-1}\Big[\frac{1}{n_2}\sum_{i=1}^{n_2}\big(f_{2,i}(u_{1,i}(s),x_2) -f_{2,i}(\bar{x}_1(s),x_2)\notag\\
  &\quad - f_{2,i}(u_{1,i}(s),x_2)\big) + L\|\bar{x}_1(s) - x_{1,i}(s)\|\Big]\notag\\
  &\le \frac{1}{t}\sum_{s=0}^{t-1}\Big[\frac{1}{n_2}\sum_{i=1}^{n_2}\big(-f_{2,i}(u_{1,i}(s),x_2) + L\|u_{1,i}(s) - \bar{x}_1(s)\|\big)\notag\\
  &\qquad + L\|\bar{x}_1(s) - x_{1,i}(s)\|\Big].\label{U_1_hat}
\end{align}
Moreover, from \eqref{sub_gradient} and the convexity of $f_{2,i}$ with respect to $x_2$,
\begin{align}
  &\quad -f_{2,i}(u_{1,i}(s),x_2)\notag\\ &= -f_{2,i}(u_{1,i}(s),x_2) + f_{2,i}(u_{1,i}(s),v_{2,i}(s))\notag\\
  &\quad - f_{2,i}(u_{1,i}(s),v_{2,i}(s))\notag\\
  &\le \langle g_{2,i}(s),v_{2,i}(s) - x_2\rangle - f_{2,i}(u_{1,i}(s),v_{2,i}(s)).\label{f_2}
\end{align}
Substituting \eqref{f_2} into \eqref{U_1_hat} yields
\begin{align}
  U(\hat{x}_{1,i}(t),x_2) &\le \frac{1}{t}\sum_{s=0}^{t-1}\frac{1}{n_2}\sum_{i=1}^{n_2}\langle g_{2,i}(s),v_{2,i}(s) - x_2\rangle\notag\\
  &\quad - \frac{1}{t}\sum_{s=0}^{t-1}\frac{1}{n_2}\sum_{i=1}^{n_2}f_{2,i}(u_{1,i}(s),v_{2,i}(s))\notag\\
  &\quad + \frac{1}{t}\sum_{s=0}^{t-1}\left[\frac{1}{n_2}\sum_{i=1}^{n_2}(L\|u_{1,i}(s) - \bar{x}_1(s)\|)\right]\notag\\
  &\quad  + \frac{1}{t}\sum_{s=0}^{t-1}\left[L\|\bar{x}_1(s) - x_{1,i}(s)\|\right].\label{U_1_hat_f}
\end{align}
Similarly,
\begin{align}
  &\quad -U(x_1,\hat{x}_{2,j}(t))\notag\\
   &\le \frac{1}{t}\sum_{s=0}^{t-1}\left[\frac{1}{n_1}\sum_{i=1}^{n_1}(-f_{1,i}(x_1,u_{2,i}(s) + L\|u_{2,i}(s) - \bar{x}_2(s)\|)\right]\notag\\
  &\quad + \frac{1}{t}\sum_{s=0}^{t-1}\left[L\|\bar{x}_2(s) - x_{2,j}(s)\|\right],
\end{align}
and
\begin{equation}
  -f_{1,i}(x_1,u_{2,i}(s)) \le \langle g_{1,i}(s),v_{1,i}(s) - x_1\rangle - f_{1,i}(v_{1,i}(s),u_{2,i}(s)).
\end{equation}
Consequently,
\begin{align}
  &\quad-U(x_1,\hat{x}_{2,j}(t))\notag\\ &\le \frac{1}{t}\sum_{s=0}^{t-1}\frac{1}{n_1}\sum_{i=1}^{n_1}\langle g_{1,i}(s),v_{1,i}(s) - x_1\rangle\notag\\
  &\quad - \frac{1}{t}\sum_{s=0}^{t-1}\frac{1}{n_1}\sum_{i=1}^{n_1}f_{1,i}(v_{1,i}(s),u_{2,i}(s))\notag\\
  &\quad + \frac{1}{t}\sum_{s=0}^{t-1}\left[\frac{1}{n_1}\sum_{i=1}^{n_1}L\|u_{2,i}(s) - \bar{x}_2(s)\|\right]\notag\\
  &\quad + \frac{1}{t}\sum_{s=0}^{t-1}[L\|\bar{x}_2(s) - x_{2,j}(s)\|].\label{U_2_hat_f}
\end{align}
It follows by Assumption \ref{asm1}(iii) that
\begin{align}
  &\quad-f_{2,i}(u_{1,i}(s),v_{2,i}(s))\notag\\ &= -f_{2,i}(u_{1,i}(s),v_{2,i}(s)) + f_{2,i}(\bar{x}_1(s),v_{2,i}(s))\notag\\
  &\quad - f_{2,i}(\bar{x}_1(s),v_{2,i}(s)) + f_{2,i}(\bar{x}_1(s),\bar{x}_2(s))\notag\\
  &\quad - f_{2,i}(\bar{x}_1(s),\bar{x}_2(s))\notag\\
  &\le L\|\bar{x}_1(s) - u_{1,i}(s)\| + L\|\bar{x}_2(s) - v_{2,i}(s)\|\notag\\
  &\quad - f_{2,i}(\bar{x}_1(s),\bar{x}_2(s)),\label{f_2_2}
\end{align}
and
\begin{align}\label{f_1_2}
  &\quad -f_{1,i}(v_{1,i}(s),u_{2,i}(s))\notag\\
  & \le L\|\bar{x}_2(s) - u_{2,i}(s)\| + L\|\bar{x}_1(s) - v_{1,i}(s)\|\notag\\
  &\quad - f_{1,i}(\bar{x}_1(s),\bar{x}_2(s)).
\end{align}
Then, since $f_1(\bar{x}_1(s),\bar{x}_2(s)) + f_2(\bar{x}_1(s),\bar{x}_2(s)) = 0$,
\begin{align}
  &\quad -\frac{1}{n_2}\sum_{i=1}^{n_2}f_{2,i}(u_{1,i}(s),v_{2,i}(s)) - \frac{1}{n_1}\sum_{i=1}^{n_1}f_{1,i}(v_{1,i}(s),u_{2,i}(s))\notag\\
  &\le \frac{1}{n_2}\sum_{i=1}^{n_2}(L\|\bar{x}_1(s) - u_{1,i}(s)\| + L\|\bar{x}_2(s) - v_{2,i}(s)\|)\notag\\
  &\quad + \frac{1}{n_1}\sum_{i=1}^{n_1}(L\|\bar{x}_2(s) - u_{2,i}(s)\| + L\|\bar{x}_1(s) - v_{1,i}(s)\|).\label{sum_1}
\end{align}
According to Lemma \ref{lem5} and $K_l\triangleq \frac{1}{\sigma_l}\left(\frac{n_lL\Gamma_l}{1 - \theta_l} + 2L\right)$, for $l = 1,2$,
\begin{align}
  \|\bar{x}_l(s) - x_{l,i}(s)\| &\le K_l\alpha + n_l\Gamma_l\theta_l^{s-1}\Lambda_l,\label{bound_3_1}\\
  \|\bar{x}_l(s) - v_{l,i}(s)\| &\le K_l\alpha + n_l\Gamma_l\theta_l^{s-1}\Lambda_l,\label{bound_4_1}\\
  \|\bar{x}_l(s) - u_{l,i}(s)\| &\le K_l\alpha + n_l\Gamma_l\theta_l^{s-1}\Lambda_l.\label{bound_5_1}
\end{align}
Applying \eqref{bound_4_1} and \eqref{bound_5_1} in \eqref{sum_1}, we obtain
\begin{align}
  & -\frac{1}{n_2}\sum_{i=1}^{n_2}f_{2,i}(u_{1,i}(s),v_{2,i}(s)) - \frac{1}{n_1}\sum_{i=1}^{n_1}f_{1,i}(v_{1,i}(s),u_{2,i}(s))\notag\\
  &\le 2L(K_1 + K_2)\alpha + 2L\sum_{l=1}^2n_l\Gamma_l\theta_l^{s-1}\Lambda_l.\label{sum_2}
\end{align}
Adding \eqref{U_1_hat_f}, \eqref{U_2_hat_f} and using \eqref{bound_3_1}, \eqref{bound_5_1}, \eqref{sum_2}, we derive
\begin{align}
  &\quad U(\hat{x}_{1,i}(t),x_2) - U(x_1,\hat{x}_{2,j}(t))\notag\\ &\le \frac{1}{t}\sum_{s=0}^{t-1}\frac{1}{n_2}\sum_{i=1}^{n_2}\langle g_{2,i}(s),v_{2,i}(s) - x_2\rangle + 4L(K_1 + K_2)\alpha\notag\\
  &\quad + \frac{1}{t}\sum_{s=0}^{t-1}\frac{1}{n_1}\sum_{i=1}^{n_1}\langle g_{1,i}(s),v_{1,i}(s) - x_1\rangle\notag\\
  &\quad + 4L\frac{1}{t}\sum_{s=0}^{t-1}\sum_{l=1}^2n_l\Gamma_l\theta_l^{s-1}\Lambda_l.\label{sum_3}
\end{align}
By \eqref{lem_bound_5}, for $l = 1,2$, for all $x_l\in\mathcal{X}_l$,
\begin{equation}\label{sum_inner}
  \frac{1}{t}\sum_{s=0}^{t-1}\frac{1}{n_l}\sum_{i=1}^{n_l}\langle g_{l,i}(s),x_{l,i}(s+1) - x_l\rangle \le \frac{\Upsilon_l^2}{t\alpha}.
\end{equation}
This together with \eqref{lem_bound_1} yields
\begin{align}
  &\quad \frac{1}{t}\sum_{s=0}^{t-1}\frac{1}{n_l}\sum_{i=1}^{n_l}\langle g_{l,i}(s),v_{l,i}(s) - x_l\rangle\notag\\
  &= \frac{1}{t}\sum_{s=0}^{t-1}\frac{1}{n_l}\sum_{i=1}^{n_l}\langle g_{l,i}(s),v_{l,i}(s) - x_{l,i}(s+1)\rangle\notag\\
  &\quad + \frac{1}{t}\sum_{s=0}^{t-1}\frac{1}{n_l}\sum_{i=1}^{n_l}\langle g_{l,i}(s),x_{l,i}(s+1) - x_l\rangle\notag\\
  &\le \frac{\Upsilon_l^2}{t\alpha} + \frac{L^2}{\sigma_l}\alpha,\label{sum_inner_f}
\end{align}
Since $(x_1^{\ast},x_2^{\ast})$ is a NE,
  \[\max_{x_2\in\mathcal{X}_2}U(\hat{x}_{1,i}(t),x_2) \ge U(\hat{x}_{1,i}(t),x_2^{\ast}) \ge U(x_1^{\ast},x_2^{\ast}),\]
  and
  \[\min_{x_1\in\mathcal{X}_1}U(x_1,\hat{x}_{2,j}(t)) \le U(x_1^{\ast},\hat{x}_{2,j}(t)) \le U(x_1^{\ast},x_2^{\ast}).\]
Substituting \eqref{sum_inner_f} into \eqref{sum_3}, we obtain
\begin{align*}
  &\quad|U(\hat{x}_{1,i}(t),\hat{x}_{2,j}(t)) - U(x_1^{\ast},x_2^{\ast})|\\ &\le \max_{x_1,x_2}\{U(\hat{x}_{1,i}(t),x_2) - U(x_1,\hat{x}_{2,j}(t))\}\\
  &\le \sum_{l=1}^2\left(\frac{\Upsilon_l^2}{t\alpha} + \frac{L^2}{\sigma_l}\alpha\right) + 4L(K_1 + K_2)\alpha\\
  &\quad + 4L\frac{1}{t}\sum_{s=0}^{t-1}\sum_{l=1}^2n_l\Gamma_l\theta_l^{s-1}\Lambda_l,
\end{align*}
which completes the proof.
\end{proof}
\begin{remark}
  Theorem \ref{thm3} shows that the cost value of the averaged iterates $U(\hat{x}_{1,i}(t),\hat{x}_{2,j}(t))$ converges to $U(x_1^{\ast},x_2^{\ast})$ within error level $(\frac{2L^2}{\sigma_1} + 4L(K_1 + K_2))\alpha$ with rate $O(\frac{1}{t} + \frac{1}{t\alpha})$, which is comparable to the rate established in \cite{nedic2009subgradient} for the centralized saddle point problems.
\end{remark}

\section{Subnetwork Zero-sum Finite-strategy Games}
In this section, we consider a subnetwork zero-sum finite-strategy game as a concrete application of Algorithm \ref{alg1}. We give the simplified algorithm in Section \uppercase\expandafter{\romannumeral 4}-A and then provide convergence results in Section \uppercase\expandafter{\romannumeral 4}-B.
\subsection{Simplified Algorithm}
For each subnetwork $\Sigma_l$ ($l\in\{1,2\}$), suppose that its (pure) action set is $\mathcal{A}_l = \{a_l^{(1)},\dots,a_{l}^{(M_l)}\}$ with $M_l$ being an integer. To obtain a continuous cost function and apply Algorithm \ref{alg1}, we consider its mixed strategy, which has also been studied in \cite{freund1999adaptive,mertikopoulos2019learning}. The mixed strategy, denoted by $x_l$, belongs to the corresponding mixed strategy set
\begin{align*}
  \mathcal{X}_l &= \Delta(\mathcal{A}_l)\\
  &:= \{x_l = (x_l^{(1)},\dots,x_l^{(M_l)})\mid \sum_{p=1}^{M_l}x_{l}^{(p)} = 1, 0\le x_l^{(p)}\le 1\}.
\end{align*}
Let $f_{1,i}(a_1^{(p)},a_2^{(q)})$ be the cost value of agent $i$ at the pure action profile $(a_1^{(p)},a_2^{(q)})$, and then the expected cost value at the mixed strategy profile $(x_1,x_2)$ is a multilinear function defined as follows.
\begin{align*}
  U(x_1,x_2) &= \frac{1}{n_1}\sum_{i=1}^{n_1}f_{1,i}(x_1,x_2)\\
  & \triangleq \frac{1}{n_1}\sum_{i=1}^{n_1}\sum_{p=1}^{M_1}\sum_{q=1}^{M_2}x_1^{(p)}x_2^{(q)}f_{1,i}(a_1^{(p)},a_2^{(q)}).
\end{align*}
Therefore, Assumption \ref{asm1} holds. Similarly, a strategy profile $x^{\ast} = (x_1^{\ast},x_2^{\ast})$ is a mixed-strategy NE of a subnetwork zero-sum finite-strategy game if \eqref{NE_def} holds.

Because $\mathcal{X}_l$ is a simplex, we consider the following negative entropy regularizer
\[\psi_l(x_l) = \sum_{p=1}^{M_l}x_l^{(p)}\log x_l^{(p)}.\]
Then $\psi_l$ is $1$-strongly convex with respect to $l_1$ norm (when there is no confusion, we use the norm $\|\cdot\|$ in Sections \uppercase\expandafter{\romannumeral 2} and \uppercase\expandafter{\romannumeral 3} as $l_1$ norm), and the Bregman divergence is given by
\[D_{\psi_l}(x_l,y_l) = \sum_{p=1}^{M_l}x_l^{(p)}\log\frac{x_l^{(p)}}{y_l^{(p)}}.\]
Therefore, Assumption \ref{asm3} holds. Moreover, through some calculations \cite{bravo2018bandit}, the update rule \eqref{update_1} can be simplified as \eqref{update_finite_1}.
The complete learning algorithm is summarized in Algorithm \ref{alg2}, which can be viewed as a distributed version of the classic multiplicative-weight (MW) algorithm \cite{freund1999adaptive}.
\begin{algorithm}[H]\caption{Distributed MW algorithm}\label{alg2}
{\it Initialize:} For $l\in\{1,2\}$, let $x_{l,i}(0) = \frac{1}{M_l}(1,\dots,1)  \in \mathcal{X}_l$.

{\it Iterate until $t\ge T$:}

\quad{\bf Communication and distributed averaging:} For $l\in\{1,2\}$,\\
\indent $\qquad$$\qquad$compute the estimates $v_{l,i}(t)$ based on \eqref{commu_v}\\
\indent $\qquad$$\qquad$compute the estimates $u_{l,i}(t)$ based on \eqref{commu_u}\\

\quad{\bf Update of $x_{l,i}(t)$:} For $l\in\{1,2\}$, $i\in\mathcal{V}_l$,\\
\indent $\qquad$$\qquad$compute the gradients $g_{l,i}(t)\in\mathbb{R}^{M_l}$:
\begin{align}
  g_{1,i}^{(p)}(t) &= f_{1,i}(a_1^{(p)},u_{2,i}(t)),\quad p = 1,\dots, M_1\label{gradient_finite_1}\\
  g_{2,i}^{(p)}(t) &= f_{2,i}(u_{1,i}(t),a_2^{(p)}),\quad p = 1,\dots, M_2\label{gradient_finite_2}
\end{align}
\indent $\qquad$$\qquad$update the estimates $x_{l,i}(t+1)$ by
\begin{equation}\label{update_finite_1}
  x_{l,i}^{(p)}(t+1) = \frac{v_{l,i}^{(p)}(t)\exp(-\alpha(t)g_{l,i}^{(p)}(t))}{\sum_{k=1}^{M_l}v_{l,i}^{(k)}(t)\exp(-\alpha(t)g_{l,i}^{(k)}(t))}, p = 1,\dots, M_l
\end{equation}
\end{algorithm}
\begin{remark}
  The update rule \eqref{update_finite_1} displays the advantage of the distributed mirror descent algorithm compared to the distributed projected subgradient descent algorithm \cite{lou2015nash}. We may avoid calculating the projection onto a simplex by choosing a negative entropy regularizer in the mirror descent algorithm.
\end{remark}
\subsection{Convergence Results}
Note that if $U$ is not strictly convex-concave, then Theorem \ref{thm2} cannot be directly applied. \cite{mertikopoulos2018optimistic} showed that even if the cost function $U$ is bilinear and admits an interior NE, the iterates of mirror descent are not convergent. Therefore, we need a stronger assumption to establish a final-iteration convergence of Algorithm \ref{alg2}. Through an analysis similar to \cite{lou2015nash}, we obtain the following convergence result.
\begin{thm}\label{thm4}
  Under Assumptions \ref{asm2} and \ref{asm4} hold, if the set of NE contains an interior point, then Algorithm \ref{alg2} generates a sequence that converges to a mixed-strategy NE for the subnetwork zero-sum finite-strategy game.
\end{thm}
\begin{proof}
  Denote by $\mathcal{X}_1^{\ast}\times\mathcal{X}_2^{\ast}$ the mixed-strategy NE set, and suppose that $(x_1^{+},x_2^{+})$ is an interior point of $\mathcal{X}_1^{\ast}\times\mathcal{X}_2^{\ast}$. Then there exists $\epsilon > 0$, such that
  \[\mathbb{B}((x_1^{+},x_2^{+}),\epsilon)\subset \mathcal{X}_1^{\ast}\times\mathcal{X}_2^{\ast}.\]
  Let $(\tilde{x}_1,\tilde{x}_2)$, $(\hat{x}_1,\hat{x}_2)$ be any two limit points of the sequence $\{\bar{x}_1(t),\bar{x}_2(t)\}$. By \eqref{lem_bound_2}, they are also the limit points of $\{x_{1,i}(t),x_{2,i}(t)\}$, and we set the corresponding convergent subsequences as $\{(x_{1,i}(t_r),x_{2,i}(t_r)\}$, $\{(x_{1,i}(t_s),x_{2,i}(t_s)\}$. Recall from the proof of Theorem 1 that the strict convexity-concavity of $U$ is not used when proving that $V(t,x_1^{\ast},x_2^{\ast})$ converges to a finite number. Thus, for all $(x_1^{\ast},x_2^{\ast})\in\mathcal{X}_1^{\ast}\times\mathcal{X}_2^{\ast}$, $V(t,x_1^{\ast},x_2^{\ast})$ still converges to a finite number. Therefore, for all $(x_1,x_2)\in\mathbb{B}((x_1^{+},x_2^{+}),\epsilon)$, by the definition of $V(t_r,x_1,x_2)$,
  \begin{align}
    D_{\psi_1}(x_1,\tilde{x}_1) + D_{\psi_2}(x_2,\tilde{x}_2) &= \lim_{r\to\infty}V(t_r,x_1,x_2)\notag\\
    & = \lim_{t\to\infty}V(t,x_1,x_2)\notag\\
    &= \lim_{s\to\infty}V(t_s,x_1,x_2)\notag\\
    & = D_{\psi_1}(x_1,\hat{x}_1) + D_{\psi_2}(x_2,\hat{x}_2).\label{equality}
  \end{align}
  Set $x_2 = x_2^{+}$ in \eqref{equality}, and then for all $x_1\in\mathbb{B}(x_1^{+},\epsilon)$,
  \[D_{\psi_1}(x_1,\tilde{x}_1) - D_{\psi_1}(x_1,\hat{x}_1) = D_{\psi_2}(x_2^{+},\hat{x}_2) - D_{\psi_2}(x_2^{+},\tilde{x}_2).\]
  Taking the derivative with respect to $x_1$ on both sides, we obtain
  \[\nabla\psi_1(x_1) - \nabla\psi_1(\tilde{x}_1) = \nabla\psi_1(x_1) - \nabla\psi_1(\hat{x}_1).\]
  Therefore, by the strong convexity of $\psi_1$,
  \[0 = \langle \nabla\psi_1(\tilde{x}_1) - \nabla\psi_1(\hat{x}_1), \tilde{x}_1 - \hat{x}_1\rangle \ge \|\tilde{x}_1 - \hat{x}_1\|_1^2.\]
  Then $\tilde{x}_1 = \hat{x}_1$. Similarly, $\tilde{x}_2 = \hat{x}_2$, i.e., the limit of $\{(\bar{x}_1(t),\bar{x}_2(t))\}$ exists. This together with \eqref{lem_bound_2} implies the existence of the limit of $\{(x_{1,i}(t),x_{2,i}(t))\}$, denoted by $(\check{x}_1,\check{x}_2)$.

  Then we prove $(\check{x}_1,\check{x}_2)\in\mathcal{X}_1^{\ast}\times\mathcal{X}_2^{\ast}$. By \eqref{network_1}, and utilizing $\sum_{t=1}^{\infty}\alpha^2(t) < \infty$ and $\sum_{t=1}^{\infty}\alpha(t)e_{1,i}(t) < \infty$, the following holds for any $x_1\in\mathcal{X}_1$,
  \[\sum_{t=0}^{\infty}\alpha(t)(U(\bar{x}_1(t),\bar{x}_2(t)) - U(x_1,\bar{x}_2(t))) < \infty.\]
Moreover, since $\sum_{t=0}^{\infty}\alpha(t) = \infty$,
  \[\lim\inf_{t\to\infty}(U(\bar{x}_1(t),\bar{x}_2(t)) - U(x_1,\bar{x}_2(t)))\le 0.\]
  The continuity of $U$ yields
  \[U(\check{x}_1,\check{x}_2) \le U(x_1,\check{x}_2),\ \forall x_1\in\mathcal{X}_1.\]
  Similarly, by \eqref{network_2},
  \[U(\check{x}_1,x_2) \le U(\check{x}_1,\check{x}_2),\ \forall x_2\in\mathcal{X}_2.\]
  By the definition of mixed-strategy NE, $(\check{x}_1,\check{x}_2)\in\mathcal{X}_1^{\ast}\times\mathcal{X}_2^{\ast}$.
\end{proof}
Moreover, motivated by Theorem \ref{thm3}, we can generalize the definition of the running average actions $\hat{x}_{l,i}(t)$ to the case of diminishing step-size, i.e.,
\[\hat{x}_{l,i}(t) = \frac{1}{\sum_{s=0}^{t-1}\alpha(s)}\sum_{s=0}^{t-1}\alpha(s)x_{l,i}(s),\quad\text{for}\ t\ge 1, l=1,2.\]
Then we obtain the following result for Algorithm \ref{alg2}.
\begin{thm}\label{thm5}
  Under Assumptions \ref{asm2} and \ref{asm4}, Algorithm \ref{alg2} generates an ergodic average sequence $(\hat{x}_{1,i}(t),\hat{x}_{2,j}(t))$ that  converges to the mixed-strategy NE set.
\end{thm}
\begin{proof}
  With analysis similar to that of Theorem \ref{thm3}, for any $(x_1,x_2)\in\mathcal{X}_1\times\mathcal{X}_2$, we obtain
  \begin{align}
    &\quad U(\hat{x}_{1,i}(t),x_2)\notag\\ &\le \frac{1}{\sum_{s=0}^{t-1}\alpha(s)}\sum_{s=0}^{t-1}\frac{1}{n_2}\sum_{i=1}^{n_2}\langle \alpha(s)g_{2,i}(s),v_{2,i}(s) - x_2\rangle\notag\\
    &\quad - \frac{1}{\sum_{s=0}^{t-1}\alpha(s)}\sum_{s=0}^{t-1}\frac{1}{n_2}\sum_{i=1}^{n_2}\alpha(s)f_{2,i}(u_{1,i}(s),v_{2,i}(s))\notag\\
    &\quad + \frac{1}{\sum_{s=0}^{t-1}\alpha(s)}\sum_{s=0}^{t-1}\alpha(s)\left[\frac{1}{n_2}\sum_{i=1}^{n_2}(L\|u_{1,i}(s) - \bar{x}_1(s)\|)\right]\notag\\
    &\quad + \frac{1}{\sum_{s=0}^{t-1}\alpha(s)}\sum_{s=0}^{t-1}\alpha(s)\left[L\|\bar{x}_1(s) - x_{1,i}(s)\|\right],\label{U_1_hat_finite}
  \end{align}
  and
  \begin{align}
    &\quad-U(x_1,\hat{x}_{2,j}(t))\notag\\ &\le \frac{1}{\sum_{s=0}^{t-1}\alpha(s)}\sum_{s=0}^{t-1}\frac{1}{n_1}\sum_{i=1}^{n_1}\langle \alpha(s)g_{1,i}(s),v_{1,i}(s) - x_1\rangle\notag\\
    &\quad - \frac{1}{\sum_{s=0}^{t-1}\alpha(s)}\sum_{s=0}^{t-1}\frac{1}{n_1}\sum_{i=1}^{n_1}\alpha(s)f_{1,i}(v_{1,i}(s),u_{2,i}(s))\notag\\
    &\quad + \frac{1}{\sum_{s=0}^{t-1}\alpha(s)}\sum_{s=0}^{t-1}\alpha(s)\left[\frac{1}{n_1}\sum_{i=1}^{n_1}(L\|u_{2,i}(s) - \bar{x}_2(s)\|)\right]\notag\\
    &\quad + \frac{1}{\sum_{s=0}^{t-1}\alpha(s)}\sum_{s=0}^{t-1}\alpha(s)\left[L\|\bar{x}_2(s) - x_{2,j}(s)\|\right].\label{U_2_hat_finite}
  \end{align}
Adding \eqref{U_1_hat_finite} and \eqref{U_2_hat_finite}, and using \eqref{sum_1} yield
\begin{align}
  &\quad U(\hat{x}_{1,i}(t),x_2) - U(x_1,\hat{x}_{2,j}(t))\notag\\ &\le \frac{1}{\sum_{s=0}^{t-1}\alpha(s)}\Bigg(\sum_{s=0}^{t-1}\frac{1}{n_2}\sum_{i=1}^{n_2}\langle \alpha(s)g_{2,i}(s),v_{2,i}(s) - x_2\rangle\notag\\
  &\quad + \sum_{s=0}^{t-1}\frac{1}{n_1}\sum_{i=1}^{n_1}\langle \alpha(s)g_{1,i}(s),v_{1,i}(s) - x_1\rangle\notag\\
  &\quad + \sum_{s=0}^{t-1}L\alpha(s)\left[\frac{1}{n_2}\sum_{i=1}^{n_2}e_{2,i}(s) + \|\bar{x}_1(s) - x_{1,i}(s)\|\right]\notag\\
  &\quad + \sum_{s=0}^{t-1}L\alpha(s)\left[\frac{1}{n_1}\sum_{i=1}^{n_1}e_{1,i}(s) + \|\bar{x}_2(s) - x_{2,j}(s)\|\right]\Bigg),\label{sum_finite}
\end{align}
with $e_{l,i}(s)$ defined in Theorem \ref{thm2}. From Lemma \ref{lem4} and Lemma \ref{lem6}, for any $(x_1,x_2)\in\mathcal{X}_1\times\mathcal{X}_2$,
\begin{align}
  &\quad\sum_{s=0}^{t-1}\frac{1}{n_l}\sum_{i=1}^{n_l}\langle \alpha(s)g_{l,i}(s),v_{l,i}(s) - x_l\rangle\notag\\
  &= \sum_{s=0}^{t-1}\frac{1}{n_l}\sum_{i=1}^{n_l}\langle \alpha(s)g_{l,i}(s),v_{l,i}(s) - x_{l,i}(s+1)\rangle\notag\\
  &\quad + \sum_{s=0}^{t-1}\frac{1}{n_l}\sum_{i=1}^{n_l}\langle \alpha(s)g_{l,i}(s),x_{l,i}(s+1) - x_l\rangle\notag\\
  &\le R_l^2 + L^2\sum_{s=0}^{t-1}\alpha^2(s).\label{inner_finite}
\end{align}
Furthermore, by \eqref{converge_cond_1}, there exists a constant $\hat{C}$ such that
\begin{equation}\label{term_3}
  \sum_{s=0}^{t-1}\alpha(s)e_{l,i}(s) \le \hat{C}\sum_{s=0}^{t-1}\alpha^2(s).
\end{equation}
Similarly, $\sum_{s=0}^{t-1}\alpha(s)\|\bar{x}_l(s) - x_{l,i}(s)\| \le \hat{C}\sum_{s=0}^{t-1}\alpha^2(s)$. Therefore, substituting \eqref{inner_finite} and \eqref{term_3} into \eqref{sum_finite}, and using Assumption \ref{asm4}, it follows
\begin{equation}\label{ergodic_converge}
  \max_{x_2\in\mathcal{X}_2}U(\hat{x}_{1,i}(t),x_2) - \min_{x_1\in\mathcal{X}_1}U(x_1,\hat{x}_{2,j}(t))\to 0,\quad t\to\infty.
\end{equation}
Consider the following gap function
\[\epsilon(\hat{x}_1,\hat{x}_2) = U^{\ast} - \min_{x_1\in\mathcal{X}_1}U(x_1,\hat{x}_{2}) + \max_{x_2\in\mathcal{X}_2}U(\hat{x}_{1},x_2) - U^{\ast},\]
where $U^{\ast}$ is the cost value of NE points. Then $\epsilon(\hat{x}_1,\hat{x}_2)\ge 0$, and the equality holds if and only if $(\hat{x}_1,\hat{x}_2)$ is a NE. By \eqref{ergodic_converge},
\[\epsilon(\hat{x}_{1,i}(t),\hat{x}_{2,j}(t))\to 0,\]
which implies that $(\hat{x}_{1,i}(t),\hat{x}_{2,j}(t))$ converges to the mixed-strategy NE set.
\end{proof}
\begin{remark}
   A similar ergodic convergence of dual averaging algorithm for finite two-person zero-sum games has been obtained in \cite{mertikopoulos2019learning}, and here we extend the result to subnetwork zero-sum games by using a distributed mirror descent algorithm.
\end{remark}
\section{SIMULATIONS}
In this section, we provide numerical examples to illustrate the no-regret property and convergence of the proposed algorithms.
\subsection{Network Interdiction}
Consider a network interdiction problem modeled by a two-player zero-sum game in \cite{washburn1995two}. We generalize it to a zero-sum game between two groups, called interdictors (${\bf I}$) and evaders (${\bf E}$). Both groups are composed of $N = 10$ agents. Group ${\bf E}$ attempts to traverse from node $s$ to node $t$ through a network $G_0$, without being detected by group ${\bf I}$. Group ${\bf I}$ selects an arc in the network and sets up an inspection site there. Denote by ${\bf I}_i$ and ${\bf E}_i$ the agent $i$ in group ${\bf I}$ and group ${\bf E}$, respectively. For agent ${\bf I}_i$, if ${\bf E}$ passes arc $k$, then he detects ${\bf E}$ with probability $p_k^i$. Similarly, for agent ${\bf E}_i$, if ${\bf I}$ selects arc $k$ to detect, he is detected with probability $q_k^i$. Denote by $P$ the set of all $s-t$ paths (i.e., the pure strategy set of ${\bf E}$) and $A$ the arc set of network $G_0$. Let $x_p$ be the probability that ${\bf E}$ selects path $p$ and $y_k$ be the probability that ${\bf I}$ selects arc $k$. The individual interdiction probabilities of group ${\bf I}$ and group ${\bf E}$ are defined as
\begin{align*}
f_{1,i}(x,y) &= \sum_{p\in P}\sum_{k\in A}x_pp_k^id_{pk}y_k = x^TA_iy,\\
f_{2,i}(x,y) &= \sum_{p\in P}\sum_{k\in A}x_pq_k^id_{pk}y_k = x^TB_iy,
\end{align*}
respectively, where $A_i = [A_{i,pk}]\triangleq[p_k^id_{pk}]$ is the payoff matrix of agent ${\bf I}_i$, $B_i = [B_{i,pk}]\triangleq[q_k^id_{pk}]$ is the cost matrix of agent ${\bf E}_i$ and $d_{pk} = 1$ if path $p$ includes arc $k$ otherwise $d_{pk} = 0$. The average probability of group ${\bf I}$ interdicting group ${\bf E}$, $\frac{1}{N}\sum_{i=1}^Nf_{1,i}(x,y)$, is equal to the average probability of ${\bf E}$ being interdicted by ${\bf I}$, $\frac{1}{N}\sum_{i=1}^Nf_{2,i}(x,y)$. Let $f(x,y)$ be the common average interdiction probability. The goal of {\bf I} is to maximize $f(\cdot,y)$, while the goal of {\bf E} is to minimize $f(x,\cdot)$. For example, in the network of Fig. 2, there are $3$ $s-t$ paths $s\to i_1\to t$, $s\to i_2\to t$, $s\to t$ and $5$ arcs $(s,i_1)$, $(i_1,t)$, $(s,i_2)$, $(i_2,t)$, $(s,t)$. The mixed strategy sets of {\bf E} and {\bf I} are $\mathcal{X} \triangleq \{x\in\mathbb{R}^3| \sum_{p=1}^3x_p = 1, x\ge 0\}$ and $\mathcal{Y} \triangleq \{y\in\mathbb{R}^5| \sum_{k=1}^5y_k = 1, y\ge 0\}$.
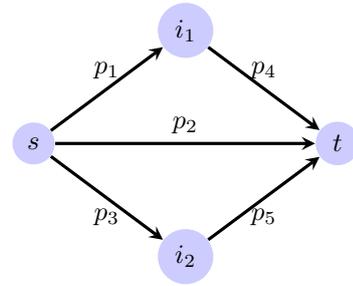
\begin{figure}[htbp]
  \centering
  \tikzstyle{node} = [circle, fill = blue!20]
  \tikzstyle{arrow} = [very thick,->,>=stealth]
  \begin{tikzpicture}[node distance=2cm]
    \node (i1) at (0,0) [node] {$i_1$};
    \node (i2) at (0,-3) [node] {$i_2$};
    \node (start) at (-2,-1.5) [node] {$s$};
    \node (target) at (2,-1.5) [node] {$t$};
    \draw [arrow] (start) -- node[above]{$p_1$}(i1);
    \draw [arrow] (i1) -- node[above]{$p_4$}(target);
    \draw [arrow] (start) -- node[above]{$p_2$}(target);
    \draw [arrow] (start) -- node[below]{$p_3$}(i2);
    \draw [arrow] (i2) -- node[below]{$p_5$}(target);
  \end{tikzpicture}
  \caption{Network interdiction problem}
\end{figure}

Agents in group ${\bf E}$ cooperate to choose a mixed strategy on $P$ while agents in  group ${\bf I}$ cooperate to choose a mixed strategy on $A$. Given a pool of connected graphs, two graphs $\mathcal{G}_1(t)$ and $\mathcal{G}_2(t)$ are randomly selected from the pool at time $t$. Suppose that each agent in groups {\bf E} and {\bf I} exchanges information with their neighbors through graphs $\mathcal{G}_1(t)$ and $\mathcal{G}_2(t)$, respectively. Furthermore, assume that agent ${\bf E}_i$ only receives information of group ${\bf I}$ from agent ${\bf I}_i$. Set $|P| = 30$, $|A| = 60$ and generate $10$ random matrices as the payoff matrices of agents in group {\bf I}. We compare the regret bound and convergence metric of Algorithm \ref{alg2} under different step-sizes.
\begin{figure}[htbp]
  \centering
  \subfigure[]{
  \begin{minipage}{8cm}
    \centering
    \includegraphics[width = .9\textwidth]{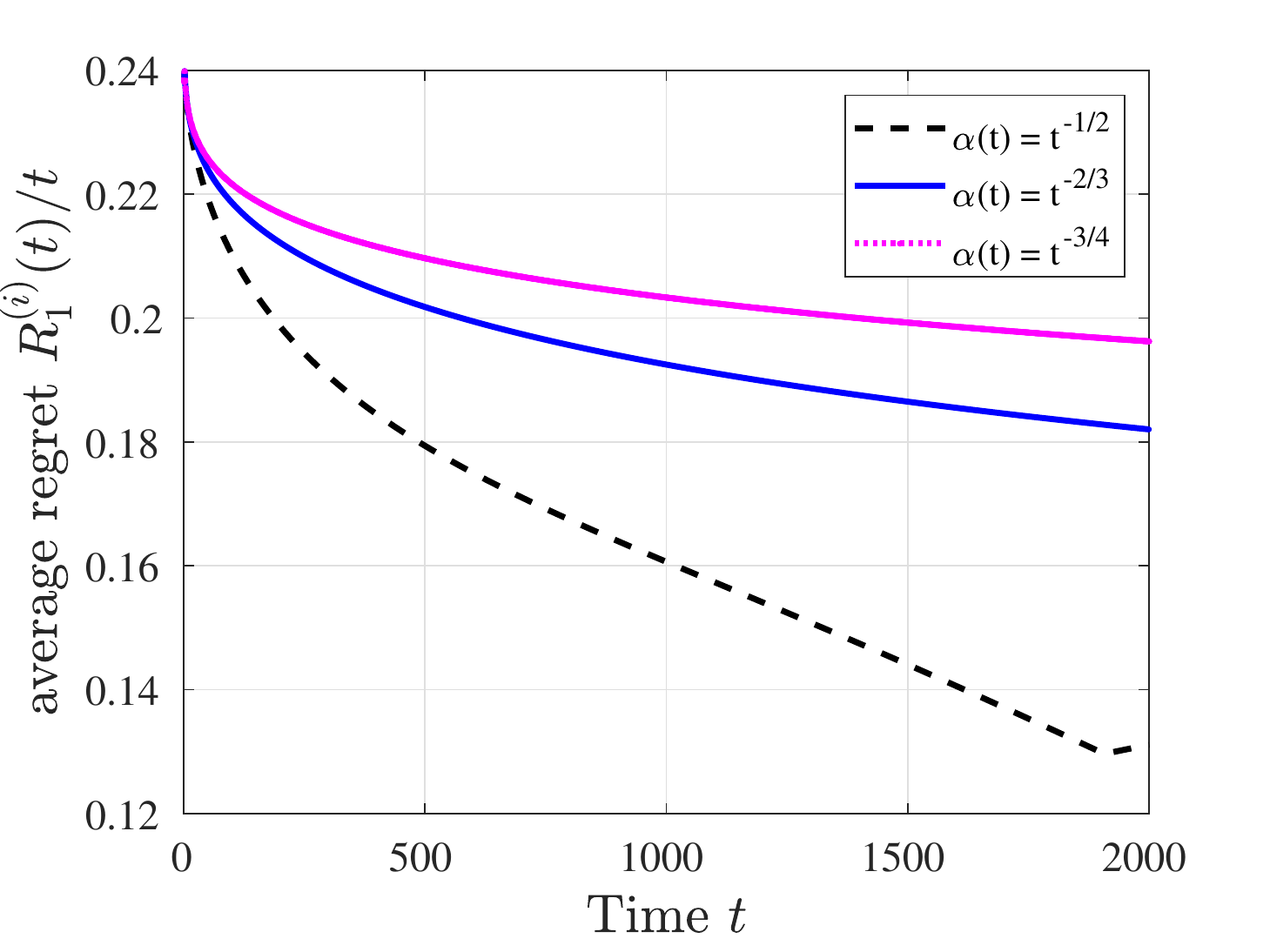}\\
  \end{minipage}
  }

  \subfigure[]{
  \begin{minipage}{8cm}
    \centering
    \includegraphics[width = .9\textwidth]{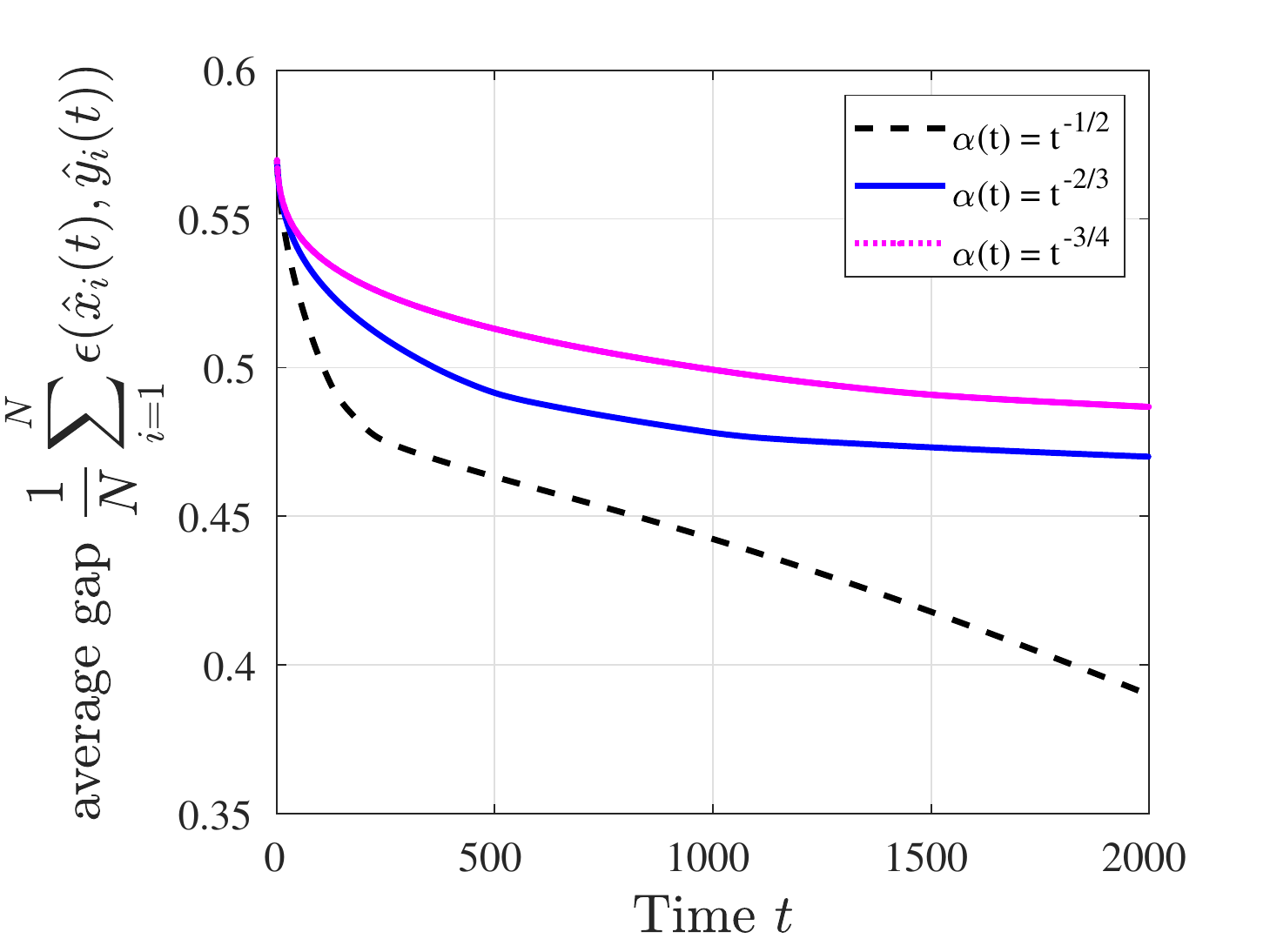}\\
  \end{minipage}
  }
  \caption{(a) Average regret of Algorithm \ref{alg2} with different step-size sequences; (b) average gap function value of Algorithm \ref{alg2} with different step-size sequences}
\end{figure}
In Fig. 3(a), we compare the average regret of agent ${\bf E}_2$ denoted by $R_1^{(2)}(T)/T$ with diminishing step-sizes $\alpha(t) = t^{-\frac{1}{2}}$, $t^{-\frac{2}{3}}$, $t^{-\frac{3}{4}}$. Algorithm \ref{alg2} produces smaller average regret when the decay rate of $\alpha(t)$ is smaller, which supports the theoretical result in Corollary \ref{col1}. Fig. 3(b) provides a plot of the average gap function
\begin{align*}
&\quad\frac{1}{N}\sum_{i=1}^N\epsilon(\hat{x}_{i}(t),\hat{y}_i(t))\\
&\triangleq \frac{1}{N}\sum_{i=1}^N\left(\max_{y\in\mathcal{Y}}\sum_{j=1}^N(\hat{x}_i(t))^TA_jy - \min_{x\in\mathcal{X}}\sum_{j=1}^Nx^TA_j\hat{y}_i(t)\right),
\end{align*}
which illustrates the ergodic convergence of Algorithm \ref{alg2}. Moreover, Fig. 3(b) also shows faster convergence under a slower diminishing step-size sequence.
\subsection{Power Allocation with Adversaries}
In this part, we use a strictly convex-concave game to verify the convergence to the unique NE. Consider a power allocation problem with adversaries over $N = 6$ Gaussian communication channels \cite{gharesifard2013distributed}. The communication rate of each channel depends on its signal power and noise power. $N$ agents connected by network $\Sigma_1$ wish to suitably allocate the total signal power among the channels so as to maximize the total communication rate, while $N$ adversaries connected by network $\Sigma_2$ attempt to minimize the total communication rate by selecting noise powers. To be specific, denote by $\{ch_1,ch_2,\dots,ch_6\}$ the channels. $\Sigma_1$ decides to allocate signal power $x_1$ to  $\{ch_1,ch_4\}$, signal power $x_2$ to $\{ch_2,ch_5\}$, and signal power $x_3$ to $\{ch_3,ch_6\}$. Similarly, $\Sigma_2$ decides to allocate noise power $y_1$ to $\{ch_1,ch_2\}$, noise power $y_2$ to $\{ch_3,ch_4\}$, and noise power $y_3$ to $\{ch_5,ch_6\}$. Both signal and noise powers satisfy a budget constraint, $2x_1 + 2x_2 + 2x_3 = 2$ and $2y_1 + 2y_2 + 2y_3 = 2$. Let $x = (x_1,x_2,x_3)$, $y = (y_1,y_2,y_3)$ and take the objective function of agent $i$ in $\Sigma_1$ as the communication rate of channel $ch_i$, which is defined by
\[f_{1,i}(x,y) = \log\left(1 + \frac{8x_{a(i)}}{\sigma(i) + y_{b(i)}}\right),\]
where $a = [1,2,3,1,2,3]$, $\sigma = [1,2,3,4,5,6]$, $b = [1,1,2,2,3,3]$. For $i\in\{1,2,3,4,5,6\}$, the individual objective function of $\Sigma_2$ is $f_{2,i}(x,y) = -f_{1,i}(x,y)$. The goal of $\Sigma_1$ is to select signal power $x$ to maximize
\[f(x,y) = \sum_{i=1}^6\log\left(1 + \frac{8x_{a(i)}}{\sigma(i) + y_{b(i)}}\right),\]
which is a strictly concave-convex function. The goal of $\Sigma_2$ is to select noise power $y$ to minimize $f(x,y)$.

For simplicity, let $\Sigma_1$, $\Sigma_2$ and $\Sigma_{12}$ be fixed networks and denote by $x^i = (x_1^i,x_2^i,x_3^i)$ ($y^i = (y_1^i,y_2^i,y_3^i)$) the estimated signal (noise) power of channel $i$ in $\Sigma_1$ ($\Sigma_2$). Take $\alpha(t) = 1/\sqrt{t}$ and the entropy regularizer $\psi_l(x) = \sum_{p=1}^3x_p\log x_p$ in Algorithm \ref{alg1}. The trajectories of the average action error of $\Sigma_1$ and the average action error of $\Sigma_2$ are plotted in Fig. 4. Here we use a centralized mirror descent method to compute the NE $(x^{\ast},y^{\ast})$. Fig. 4 shows that $(x^i,y^j)$ converges to $(x^{\ast},y^{\ast})$, and thus, verifies Theorem \ref{thm2}.
\begin{figure}[htbp]
\centering
\includegraphics[width = .4\textwidth]{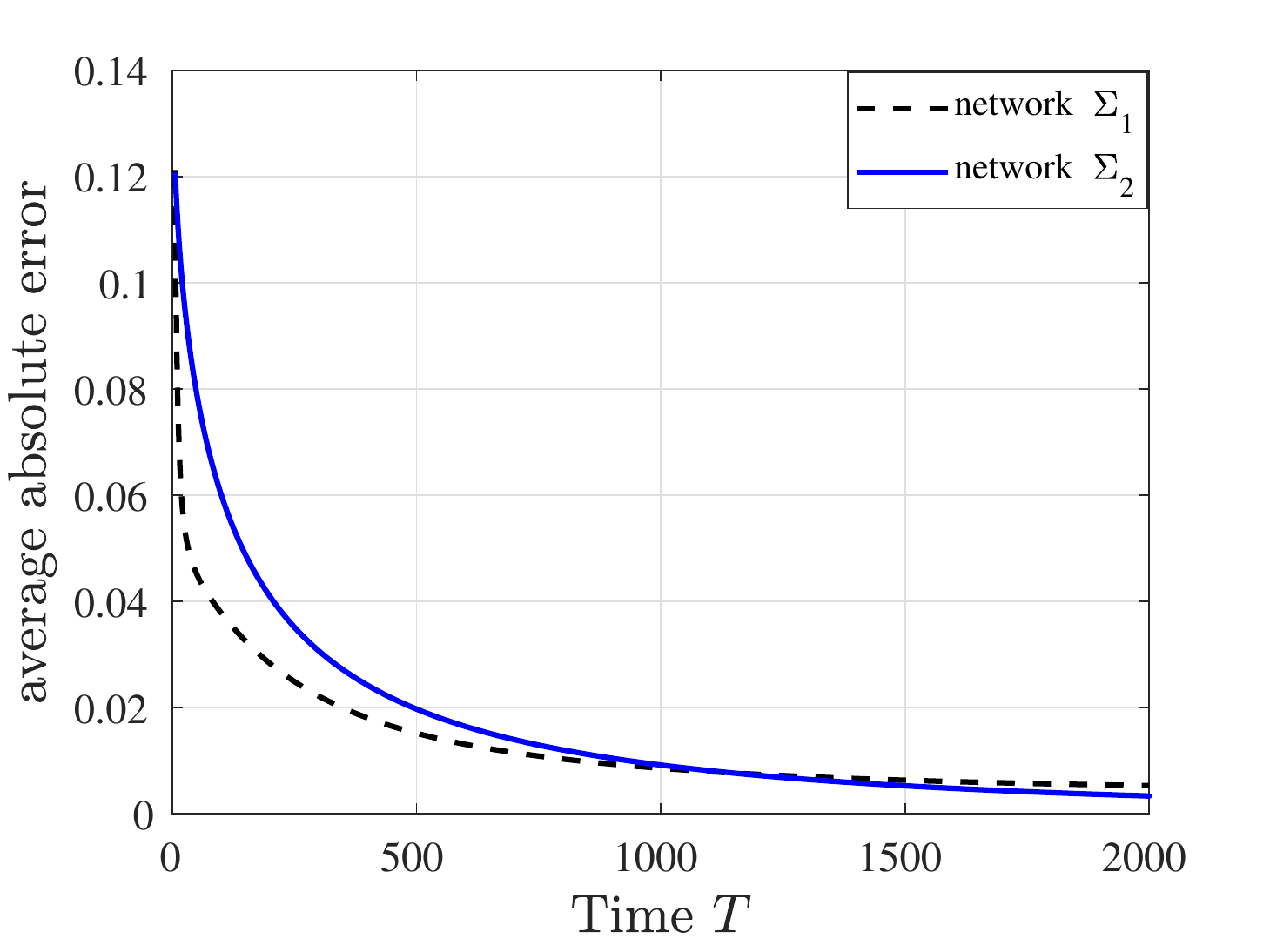}
\caption{Average absolute error $\frac{1}{N}\sum_{i=1}^N\|x^i(T) - x^{\ast}\|$ and $\frac{1}{N}\sum_{i=1}^N\|y^i(T) - y^{\ast}\|$}
\end{figure}

Let us consider the effect of the algebraic connectivity of the communication network on the regret bound. Denote the algebraic connectivity of $\Sigma_1$ as $\lambda_2$, which is defined as the second smallest eigenvalue of its Laplacian matrix. With $\Sigma_2$ and $\Sigma_{12}$ unchanged, the trajectories of the average regret of channel 1 
for $\lambda_2\in\{0.4,0.8,1.2\}$ are plotted in Fig. 5. Fig. 5 shows that $\lambda_2$ does not change the rate of average regret, while Algorithm \ref{alg1} produces a smaller regret if $\Sigma_1$ has a larger algebraic connectivity.
\begin{figure}[htbp]
\centering
\includegraphics[width = .4\textwidth]{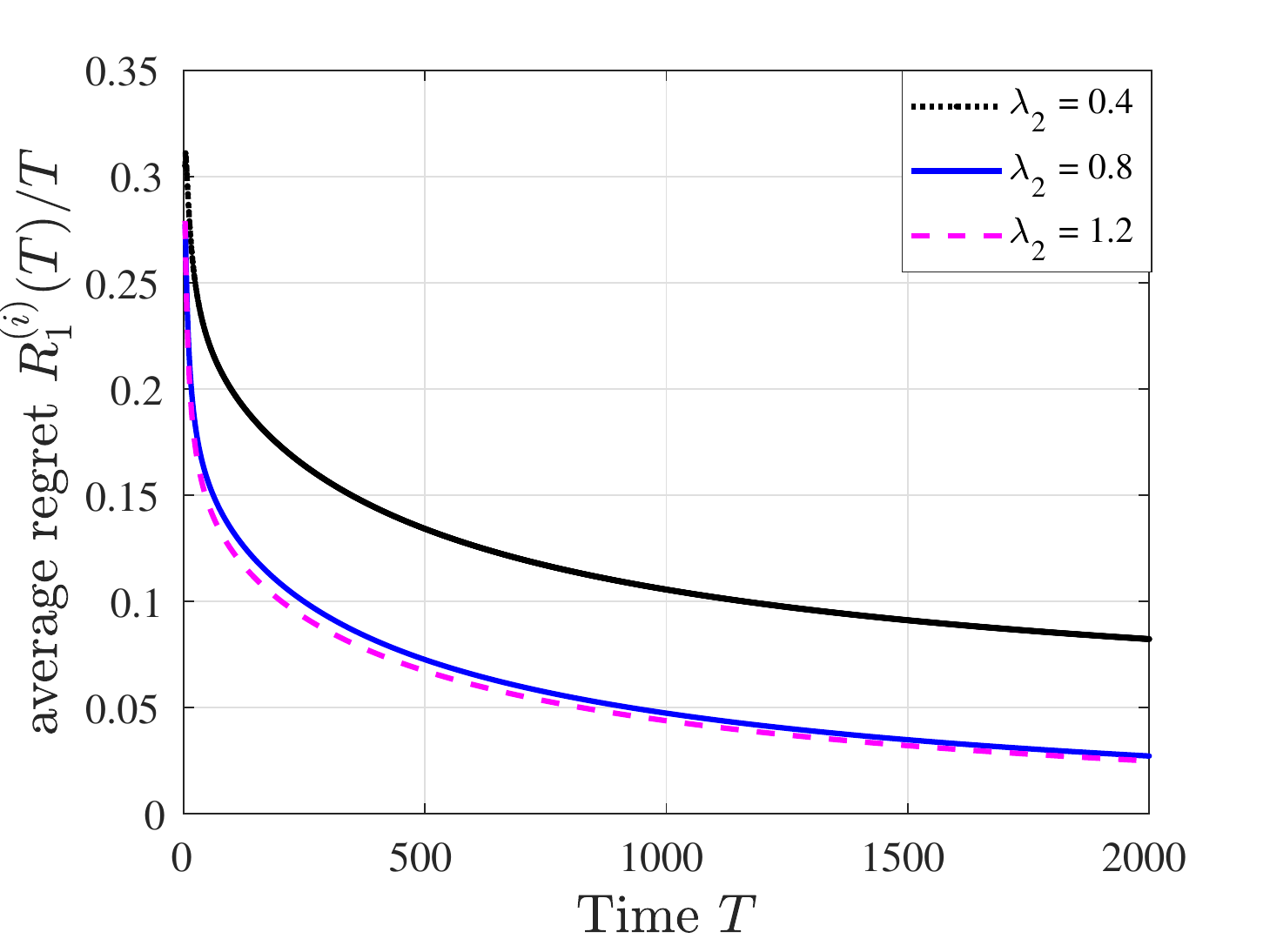}
\caption{Average regret of channel 1 for $\lambda_2\in\{0.4,0.8,1.2\}$}
\end{figure}

\section{CONCLUSION}
In this paper, we proposed a distributed mirror descent algorithm for NE seeking in a subnetwork zero-sum game. First, we provide a regret analysis for the proposed algorithm under both diminishing and constant step-sizes. We also prove its convergence to the NE under diminishing step-sizes. Our analysis demonstrates that the proposed algorithm satisfies a no-regret property while converging to the NE. Moreover, we establish an asymptotic error bound on the cost value of averaged iterates in the constant step-size case. Finally, we prove a final-iteration convergence result and an ergodic convergence result, respectively, under diverse assumptions on the cost functions in subnetwork zero-sum finite-strategy games.

\section*{Appendix}
{\it Proof of Lemma 4}. Applying the optimality condition of \eqref{update_1} and recalling \eqref{breg_def}, we have that for each $l = 1,2$, and any $x_l\in\mathcal{X}_l$,
\begin{equation}\label{opt_cond_1}
  \langle \nabla \psi_l(x_{l,i}(t+1)) - \nabla \psi_l(v_{l,i}(t)) + \alpha(t)g_{l,i}(t), x_l - x_{l,i}(t+1)\rangle \ge 0.
\end{equation}
By setting $x_l = v_{l,i}(t)$ in \eqref{opt_cond_1}, we obtain
\begin{align}
  &\Big\langle \nabla \psi_l(x_{l,i}(t+1)) - \nabla \psi_l(v_{l,i}(t)) + \alpha(t)g_{l,i}(t), \notag\\
  &\ \  x_{l,i}(t+1) - v_{l,i}(t)\Big\rangle\le 0.\label{proj}
\end{align}
Therefore, from the strong convexity of $\psi_l$,
\begin{align*}
  &\quad \alpha(t)\|g_{l,i}(t)\|_{\ast}\|v_{l,i}(t) - x_{l,i}(t+1)\|\\ &\ge \langle\alpha(t)g_{l,i}(t), v_{l,i}(t) - x_{l,i}(t+1)\rangle\\
  &\ge \langle \nabla \psi_l(x_{l,i}(t+1)) - \nabla \psi_1(v_{l,i}(t)), x_{l,i}(t+1) - v_{l,i}(t)\rangle\\
  &\ge \sigma_l\|v_{l,i}(t) - x_{l,i}(t+1)\|^2,
\end{align*}
which together with \eqref{bound_subgradient} yields \eqref{lem_bound_1}.\\
{\it Proof of Lemma 5}. By \eqref{commu_v} and defining $p_{l,i}(t-1) = x_{l,i}(t) - v_{l,i}(t-1)$, we obtain
\begin{align*}
  x_{l,i}(t) &= v_{l,i}(t-1) - (v_{l,i}(t-1) - x_{l,i}(t))\notag\\
  &= \sum_{j\in\mathcal{N}_{l,i}(t-1)}w_{l,ij}(t-1)x_{l,j}(t-1) + p_{l,i}(t-1)\notag\\
  &= \sum_{j=1}^{n_l}[\Phi_l(t-1,0)]_{ij}x_{l,j}(0)\\
  &\quad + \sum_{s=1}^{t-1}\sum_{j=1}^{n_l}[\Phi_l(t-1,s)]_{ij}p_{l,j}(s-1) + p_{l,i}(t-1).
\end{align*}
 Since $W_l(t)$ is doubly stochastic,
\begin{align*}
  \bar{x}_l(t) &= \frac{1}{n_l}\sum_{i=1}^{n_l}x_{l,i}(t)\\
  &\quad = \frac{1}{n_l}\sum_{j=1}^{n_l}x_{l,j}(0) + \frac{1}{n_l}\sum_{s=1}^t\sum_{j=1}^{n_l}p_{l,j}(s-1).
\end{align*}
Therefore, Lemma \ref{lem3} and \eqref{lem_bound_1} together yield
\begin{align}
  &\quad \|x_{l,i}(t) - \bar{x}_l(t)\|\notag\\ &\le \sum_{j=1}^{n_l}\bigg|[\Phi_l(t-1,0)]_{ij} - \frac{1}{n_l}\bigg|\|x_{l,j}(0)\|\notag\\
  &\quad + \sum_{s=1}^{t-1}\sum_{j=1}^{n_l}\bigg|[\Phi_l(t-1,s)]_{ij} - \frac{1}{n_l}\bigg|\|p_{l,j}(s-1)\|\notag\\
  &\quad + \|\frac{1}{n_l}\sum_{j=1}^{n_l}p_{l,j}(t-1) - p_{l,i}(t-1)\|\notag\\
  &\le n_l\Gamma_l\theta_l^{t-1}\Lambda_l\notag\\
  &\quad + \frac{1}{\sigma_l}\left(n_lL_{l,1}\Gamma_l\sum_{s=1}^{t-1}\theta_l^{t-1-s}\alpha(s-1) + 2L_{l,1}\alpha(t-1)\right).\notag
\end{align}
Thus, \eqref{lem_bound_2} holds. Furthermore, by $\sum_{j=1}^{n_l}w_{l,ij}(t) = 1$,
\begin{align*}
  &\quad\|v_{l,i}(t) - \bar{x}_{l}(t)\|\\
   &\overset{\eqref{commu_v}}{ =} \|\sum_{j=1}^{n_l}w_{l,ij}(t)x_{l,j}(t) - \bar{x}_{l}(t)\|\notag\\
  &\le \sum_{j=1}^{n_l}w_{l,ij}(t)\|x_{l,j}(t) - \bar{x}_{l}(t)\|\notag\\
  &\overset{\eqref{lem_bound_2}} {\le} n_l\Gamma_l\theta_l^{t-1}\Lambda_l\\
  &\quad + \frac{1}{\sigma_l}\left(n_lL_{l,1}\Gamma_l\sum_{s=1}^{t-1}\theta_l^{t-1-s}\alpha(s-1) + 2L_{l,1}\alpha(t-1)\right),
\end{align*}
Thus, \eqref{lem_bound_3} holds. Similarly, by $\sum_{j=1}^{n_l}w_{12,ij}(t) = 1$ and \eqref{lem_bound_2}, we obtain \eqref{lem_bound_4}.\\
{\it Proof of Lemma 6}. By setting $x_l = \breve{x}_l$ in \eqref{opt_cond_1} and rearranging the terms, we obtain
\begin{align*}
  &\quad \langle \alpha(t)g_{l,i}(t),x_{l,i}(t+1) - \breve{x}_l\rangle\\ &\le \langle \nabla\psi_l(v_{l,i}(t)) - \nabla\psi_l(x_{l,i}(t+1)),x_{l,i}(t+1) - \breve{x}_l\rangle\\
  &= D_{\psi_l}(\breve{x}_l, v_{l,i}(t)) - D_{\psi_l}(\breve{x}_l, x_{l,i}(t+1))\\
  &\quad - D_{\psi_l}(x_{l,i}(t+1), v_{l,i}(t))\\
  &\le D_{\psi_l}(\breve{x}_l, v_{l,i}(t)) - D_{\psi_l}(\breve{x}_l, x_{l,i}(t+1))\\
  &\le \sum_{j=1}^{n_l}w_{l,ij}(t)D_{\psi_l}(\breve{x}_l, x_{l,j}(t)) - D_{\psi_l}(\breve{x}_l, x_{l,i}(t+1)),
\end{align*}
where the equality follows from \eqref{breg_prop1} with $x = v_{l,i}(t)$, $y = \breve{x}_l$, $z = x_{l,i}(t+1)$, the second inequality holds since $D_{\psi_l}(x_{l,i}(t+1), v_{l,i}(t))\ge 0$ by \eqref{breg_prop2}, and the last inequality follows from Assumption \ref{asm3} and Jensen's inequality since $v_{l,i}(t) = \sum_{j=1}^{n_l}w_{l,ij}(t)x_{l,j}(t)$ and $\sum_{j=1}^{n_l}w_{l,ij}(t) = 1$. Therefore, from Assumption \ref{asm1}(i) and $\sum_{i=1}^{n_l}w_{l,ij}(t) = 1$,
\begin{align}
  &\quad\frac{1}{n_l}\sum_{t=1}^T\sum_{i=1}^{n_l}\langle g_{l,i}(t), x_{l,i}(t+1) - \breve{x}_l\rangle\notag\\
  &\le \frac{1}{n_l}\sum_{t=1}^T\frac{1}{\alpha(t)}\Big[\sum_{i=1}^{n_l}(\sum_{j=1}^{n_l}w_{l,ij}(t)D_{\psi_l}(\breve{x}_l, x_{l,j}(t))\notag\\
  &\qquad\qquad\quad\quad - D_{\psi_l}(\breve{x}_l, x_{l,i}(t+1)))\Big]\notag\\
  &= \frac{1}{n_l}\sum_{i=1}^{n_l}\sum_{t=1}^T\frac{1}{\alpha(t)}[D_{\psi_l}(\breve{x}_l, x_{l,i}(t)) - D_{\psi_l}(\breve{x}_l, x_{l,i}(t+1))]\notag\\
  &\le \frac{1}{n_l}\sum_{i=1}^{n_l}\Big[\frac{1}{\alpha(1)}D_{\psi_l}(\breve{x}_l, x_{l,i}(1))\notag\\
  &\qquad + \sum_{t=2}^TD_{\psi_l}(\breve{x}_l, x_{l,i}(t))\left(\frac{1}{\alpha(t)} - \frac{1}{\alpha(t-1)}\right)\Big]\notag\\
  &\le \frac{1}{n_l}\sum_{i=1}^{n_l}\left[\frac{\Upsilon_l^2}{\alpha(1)} + \sum_{t=2}^T\left(\frac{\Upsilon_l^2}{\alpha(t)} - \frac{\Upsilon_l^2}{\alpha(t-1)}\right)\right] \le \frac{\Upsilon_l^2}{\alpha(T)}.\label{bound_2}
\end{align}
\bibliographystyle{IEEEtran}
\bibliography{no_regret_distributed_learning.bib}

\end{document}